\documentclass[12pt,a4paper]{article}

\usepackage{graphicx}
\usepackage{amsmath,amssymb,bm,amsthm,mathrsfs,amscd, mathtools, stmaryrd,url}
\usepackage{here}
\usepackage{fancyhdr}
\usepackage{tikz}
\usepackage{xcolor}
\usepackage{authblk}
\usepackage[hidelinks]{hyperref}
\usepackage[inline]{enumitem}
\usepackage{wrapfig}
\usetikzlibrary{calc}

\theoremstyle{plain}
\newtheorem{thm}{Theorem}[section]
\newtheorem{lem}[thm]{Lemma}
\newtheorem{cor}[thm]{Corollary}
\newtheorem{prop}[thm]{Proposition}

\newtheorem{claim}{Claim}

\theoremstyle{definition}
\newtheorem{dfn}[thm]{Definition}
\newtheorem{eg}[thm]{Example}

\theoremstyle{remark}
\newtheorem{rem}[thm]{Remark}

\numberwithin{equation}{section}

\allowdisplaybreaks

\DeclareMathOperator{\Ann}{Ann}
\DeclareMathOperator{\AG}{AG}

\DeclareMathOperator{\Cyc}{Cyc}
\DeclareMathOperator{\diag}{diag}

\DeclareMathOperator{\Ker}{Ker}

\DeclareMathOperator{\Min}{Min}
\DeclareMathOperator{\ord}{ord}
\DeclareMathOperator{\pr}{pr}

\DeclareMathOperator{\rowspan}{rowspan}
\DeclareMathOperator{\supp}{supp}
\DeclareMathOperator{\width}{width}

\newcommand{\abs}[1]{\left\lvert #1 \right\rvert}
\renewcommand{\Im}{\operatorname{Im}}

 \title{On Representing Matroids via Modular Independence}
\author[1]{Koji Imamura\thanks{k-imamura@kumamoto-u.ac.jp}}
\author[2]{Keisuke Shiromoto\thanks{keisuke@kumamoto-u.ac.jp}}

\affil[1]{Research and Education Institute for Semiconductors and Informatics, Kumamoto University}
\affil[2]{Graduate School of Science and Technology, Kumamoto University}
\date{} 
\begin{document}
  \maketitle
  \begin{abstract}
  We study a matrix-based notion of matroid representation
  over local commutative rings obtained by replacing 
  linear independence with modular independence.
  This construction always defines an independence system, though not necessarily a matroid.
  Under a mild nilpotent hypothesis,
  we show that chain rings are exactly the local rings
  for which the minimal number of generators is monotone
  on finitely generated submodules, and over commutative
  chain rings we obtain a criterion for the associated independence
  system to be a matroid.
  For codes over finite commutative chain rings,
  we identify puncturing with deletion,
  show that shortening agrees with contraction under a contractibility
  hypothesis, and establish duality for free codes. 
  We further derive bounds for simple and uniform matroids,
  prove that the uniform matroid $U_{2, n}$ is representable if and only if
  the size $n$ is at most the sum of the cardinalities of the local ring
  and its unique maximal ideal, and show that all excluded minors for
  $\mathbb{F}_{4}$-representability are representable over $\mathbb{Z}/4\mathbb{Z}$.
  The examples also include ring representations of matroids
  not representable over any field, such as the V\'{a}mos matroid
  over $\mathbb{Z}/8\mathbb{Z}$.
\end{abstract} 
\section{Introduction} \label{sect:introduction}

The matroid representation problem asks whether a given matroid $M$
is representable over a field $\mathbb{F}$, i.e.,
whether there exists a matrix $A$ over $\mathbb{F}$ whose
linearly independent sets of columns are exactly
the independent sets of $M$.
Although matroids abstract linear independence over a field,
P.~Nelson \cite{Nelson18} proved that almost all matroids are
not representable over any field.
A central theme in classical matroid theory is to determine,
for a fixed field $\mathbb{F}$, which matroids are $\mathbb{F}$-representable.
For the smallest finite fields, this problem admits striking excluded-minor
characterisations.
Binary matroids are exactly those with no $U_{2, 4}$-minor.
Ternary matroids are characterised by the excluded minors $U_{2, 5}$,
$U_{3, 5}$, $F_{7}$, and $F_{7}^{\ast}$.
The characterisation of $\mathbb{F}_{4}$-representable matroids,
a landmark result of Geelen, Gerards, and Kapoor \cite{ExcludedMinorGF4},
is achieved by the seven excluded minors $U_{2, 6}$, $U_{4, 6}$, $P_{6}$,
$F_{7}^{-}$, $(F_{7}^{-})^{\ast}$, $P_{8}$, and $P_{8}^{=}$.
These characterisations provide a natural benchmark for the comparison
in Section~\ref{sect:application}.
For larger fields, however, the excluded-minor discussion becomes
much less explicit; for instance, representability over $\mathbb{F}_{5}$
already leads to a substantially more complicated theory than
in the binary, ternary, and quaternary cases; see \cite{GGW14,Brettell25}.
For more details on classical matroid representation problems, see \cite[Section~6.5]{Oxley}.
Beyond matrix representations over fields and a classical alternative
based on algebraic independence, recent work has explored
representations over hyperfields \cite{BB18, BB19}.

In this paper, we extend the matrix-based construction of vector matroids
to the case of rings via modular independence.
Modular independence is introduced by Park \cite{Park09} for the ring $\mathbb{Z}_{m}$
of integers modulo $m$, and generalised to the class of Frobenius rings
by Dougherty and Liu \cite{DoughertyLiu09}. In this paper, we begin with
a local commutative (not necessarily Frobenius) ring,
where the definition of modular independence coincides with
that used for Frobenius rings.
This framework yields matrix representations over rings
for several matroids that are not representable over any field.

Throughout, “matroids over a ring” refers to our matrix-based framework
using modular independence. This should not be confused with
the module-valued notion of matroids over a ring introduced by
Fink and Moci \cite{FinkMoci16}.
Fink-Moci's ``matroids over a ring'' are a module-valued
generalisation of matroids: to each subset $X \subseteq E$ one assigns
a finitely generated $R$-module $M(X)$ satisfying exactness/rank-like
axioms, rather than a linear representation.
The framework specialises to ordinary matroids when $R$ is a field,
to arithmetic matroids for $R = \mathbb{Z}$, and to valuated/tropical
data for discrete valuation rings (DVRs) via module length/valuation.
Their work suggests some potential applications to linear codes
over a ring, though this connection is not developed in their paper.

Motivated by codes over rings, several matroid-like formalisms
capture dependence, support, and weight for module-linear codes.
Vertigan \cite{Vertigan04} introduced latroids, which associate
to a (module-)linear code a combinatorial object whose
Tutte polynomial determines the code's weight enumerator,
extending the matroidal viewpoint to modules (see also \cite{GorlaSalizzoni25}).
Demi-matroids introduced by Britz, Johnsen, Mayhew, and Shiromoto
\cite{WeiTypeDuality} relax the rank axioms and furnish the right setting
for generalised Hamming weights and Wei-type duality,
thereby accommodating codes beyond fields (see also \cite{WeiTypeDuality, HarmonicTutte2}).

Our aim is to extend the matrix construction of vector matroids
from fields to rings via modular independence;
we do not pursue axiomatic generalisations in this paper.
At the level of general local rings, modular independence
always determines an independence system,
whereas the stronger matroidal properties studied in this paper
emerge most naturally over chain rings.
Accordingly, Section~\ref{sect:matroid_repr_by_modular_indep}
works over local rings, while Sections~\ref{sect:matroid_repr_codes} and
\ref{sect:application} specialise to finite commutative chain rings.

The organisation of the paper is as follows.
In Section~\ref{sect:preliminaries}, we review the necessary background
on matroids, independence systems, clutters, local and chain rings,
linear codes over rings, and modular independence.
In Section~\ref{sect:matroid_repr_by_modular_indep}, we investigate
when modular independence on subsets of modules gives rise to matroids.
In particular, we characterise chain rings by the property that
the minimal number of generators of a finitely generated submodule
is monotone under inclusion, and over chain rings we obtain
a criterion for the submodularity of the resulting rank function.
In Section~\ref{sect:matroid_repr_codes}, we study the independence systems
arising from codes over finite chain rings.
We identify puncturing with deletion, prove that shortening
agrees with contraction under a contractibility hypothesis,
and establish duality for free codes.
In Section~\ref{sect:application}, we study representations
over finite chain rings more concretely.
We derive size bounds for simple matroids and uniform matroids,
compare representability over $\mathbb{Z}/4\mathbb{Z}$ and
$\mathbb{Z}/8\mathbb{Z}$ with representability over fields
of the same cardinality, and show that all excluded minors for
$\mathbb{F}_{4}$-representability admit representations over
$\mathbb{Z}/{4}\mathbb{Z}$.
Appendix~\ref{sect:appendix} collects explicit matrices
for the examples used in Section~\ref{sect:application}. 
\section{Preliminaries} \label{sect:preliminaries}

\subsection{Matroids, Independence Systems, and Clutters} \label{subsect:matroid_indep_sys_clutter}

A \emph{matroid} $M$ is defined as an ordered pair $(E, \mathcal{I})$
consisting of a finite set $E$ and a collection $\mathcal{I}$ of
subsets of $E$ satisfying the following three properties:
\begin{enumerate}[label=\textup{(I\arabic{*})}, itemsep=0ex]
  \item $\emptyset \in \mathcal{I}$;
  \item if $I \in \mathcal{I}$ and $I^{\prime} \subseteq I$, then $I^{\prime} \in \mathcal{I}$;
  \item for all $I_{1}, I_{2} \in \mathcal{I}$ with $\abs{I_{1}} < \abs{I_{2}}$,
          there is some element $e \in I_{2} \setminus I_{1}$ such that
          $I_{1} \cup \{ e \} \in \mathcal{I}$.
\end{enumerate}

The pair $M = (E, \mathcal{I})$ satisfying the first two properties
(I1) and (I2) is called an \emph{independence system}.
Even if $M$ is an independence system but not a matroid,
we call the set $E$ the \emph{ground set}, and each member of $\mathcal{I}$
an \emph{independent set} of $M$.
A subset of $E$ which is not in $\mathcal{I}$ is called
a \emph{dependent set} of $M$. An independence system is
primarily a term in combinatorial optimisation;
in the context of combinatorial topology,
$M$ is also called an \emph{abstract simplicial complex} with \emph{vertex set} $E$.
(Strictly speaking, to call $E$ the vertex set, it is necessary
that $\{ e \} \in \mathcal{I}$ for every $e \in E$.)
In the context, each member of $\mathcal{I}$ is called a \emph{simplex}.
In this paper, we adopt the term ``independence system''
in order to ensure consistency of terminology and notation.

One of the most attractive features of matroids is that
there are many equivalent axiom systems for matroids,
which is known as \emph{cryptomorphisms} (see, for example \cite[Chapter 2]{GordonMcNulty12}).
Actually, if $(E, \mathcal{I})$ is a matroid, the integer-valued function
$r \colon 2^{E} \to \mathbb{Z}_{\geq 0}$ defined as
\begin{equation*}
  r(X) \coloneqq \max \{ \abs{I} : I \subseteq X, \, I \in \mathcal{I} \} \quad \text{for all } X \subseteq E
\end{equation*}
satisfies the following three properties:
\begin{enumerate}[label=\textup{(R\arabic{*})},itemsep=0ex]
  \item For all $X \subseteq E$, $0 \leq r(X) \leq \abs{X}$;
  \item For all $X \subseteq Y \subseteq E$, $r(X) \leq r(Y)$;
  \item For all $X, Y \subseteq E$, $r(X \cup Y) + r(X \cap Y) \leq r(X) + r(Y)$.
\end{enumerate}
The last condition (R3) is called \emph{submodularity}.
The function $r$ is called the rank function of the matroid.
Conversely, given a function $r \colon 2^{E} \to \mathbb{Z}_{\geq 0}$
satisfying (R1)--(R3),
\begin{equation*}
  \mathcal{I} \coloneqq \{ I \subseteq E \mid r(I) = \abs{I} \}
\end{equation*}
satisfies (I1)--(I3), and thus $(E, \mathcal{I})$ becomes a matroid.
Furthermore, we can similarly characterise the independence system
as the function $r \colon 2^{E} \to \mathbb{Z}_{\geq 0}$
with the condition (R1), (R2), and
\begin{enumerate}[itemsep=0ex]
  \item[] (R2+) For all $X \subseteq E$ and $e \in E$, $r(X \cup \{ e \}) \leq r(X) + 1$.
  \item[] (H) For all $I \subseteq E$ with $r(I) = \abs{I}$ and for all $e \in I$, $r(I \setminus \{ e \}) = \abs{I} - 1$.
\end{enumerate}
Note that if (R2+) holds, we may replace (R1) with (R1')
$r(\emptyset) = 0$. In fact, if (R1') and (R2+) hold,
$r$ satisfies (R1) because $0 = r(\emptyset) \leq r(\emptyset \cup X) = r(X) \leq r(\emptyset) + \abs{X} = \abs{X}$.

If $(E, \mathcal{I})$ is a matroid, the collection of minimal dependent sets
(under the inclusion relation),
\begin{equation} \label{eq:cryptomorphism_I_to_C}
   \mathcal{C} = \{ D \subseteq E \mid D \notin \mathcal{I} \text{ and  } I \in \mathcal{I} \text{ for all } I \subsetneq D \}
\end{equation}
satisfies the following three properties:
\begin{enumerate}[label=\textup{(C\arabic{*})}, itemsep=0ex]
  \item $\emptyset \notin \mathcal{C}$;
  \item if $D_{1}, D_{2} \in \mathcal{C}$ and $D_{1} \subseteq D_{2}$, then $D_{1} = D_{2}$;
  \item if $D_{1}, D_{2} \in \mathcal{C}$ with $D_{1} \neq D_{2}$,
        and $x \in D_{1} \cap D_{2}$, then $D_{3} \subseteq (D_{1} \cup D_{2}) \setminus \{ x \}$
        for some $D_{3} \in \mathcal{C}$.
\end{enumerate}
Conversely, if $\mathcal{C} \subseteq 2^{E}$ satisfies (C1)--(C3), the collection
\begin{equation*}
  \mathcal{I}^{\prime} = \{ I \subseteq E \mid D \nsubseteq I, \text{for all } D \in \mathcal{C} \}
\end{equation*}
satisfies (I1)--(I3), and then $\mathcal{I}^{\prime} = \mathcal{I}$.
Thus, we may define a matroid as the pair $(E, \mathcal{C})$,
where $\mathcal{C}$ is a collection of subsets of $E$ satisfying
(C1), (C2), and (C3). Then each member of $\mathcal{C}$ is called
a \emph{circuit} of the matroid. Furthermore, it is straightforward that
the independence system is similarly characterised as
the pair $(E, \mathcal{C})$ with the condition (C1) and (C2).
Such $\mathcal{C}$ forms an \emph{antichain} in the Boolean lattice,
and it is also called a \emph{clutter} or \emph{Sperner family}.

In \cite{CluttersMatroids}, the following two operations called
\emph{deletion} and \emph{contraction} are defined on
the collection of the circuits of an independence system.
We note that they introduced two different types of
deletion and contraction for clutters.
One respects the collection of bases of matroids \cite[Definition 2.1]{CluttersMatroids},
and the other is consistent with the collection of circuits of matroids \cite[Definition 4.1]{CluttersMatroids}.
In terms of abstract simplicial complex, they correspond to
deletion and link, respectively.
For a collection $\mathcal{A}$ of subsets of $E$,
we denote the set of minimal members in $\mathcal{A}$
under the inclusion relation by $\Min(\mathcal{A})$.

\begin{dfn}[{\cite[Definition~4.1]{CluttersMatroids}}] \label{dfn:deletion_and_contraction}
  Let $(E, \mathcal{C})$ be an independence system, namely,
  $\mathcal{C}$ satisfies (C1) and (C2).
  The element $e \in E$ is a \emph{loop} (resp. \emph{isthmus} or a \emph{coloop})
  if $\{ e \} \in \mathcal{C}$ (resp. $e \notin D$ for all $D \in \mathcal{C}$).
  The \emph{deletion} of $T \subseteq E$ from $\mathcal{C}$ is
  the collection of subsets of $E \setminus T$ defined by
  \begin{equation*}
    \mathcal{C} \backslash T \coloneqq \{ D \subseteq E \setminus T \mid D \in \mathcal{C} \}.
  \end{equation*}
  The \emph{contraction} of $T \subseteq E$ from $\mathcal{C}$ is the collection of subsets of $E \setminus T$ defined by
  \begin{equation*}
    \mathcal{C}/T \coloneqq \Min(\{ D \setminus T \mid D \in \mathcal{C} \}).
  \end{equation*}
  In the clutter-theoretic definition above,
  contracting a loop would produce the empty set
  and hence violate \textnormal{(C1)}.
  Accordingly, whenever $\mathcal{C}/T$ is used below,
  we implicitly assume that $T$ is independent
  in the associated independence system;
  equivalently, $T$ contains no member of $\mathcal C$.
  A collection of circuits which is obtained by a sequence of deletions
  and contractions from $\mathcal{C}$ is called a \emph{minor} of $\mathcal{C}$.
\end{dfn}

It is clear that $\mathcal{C}/T$ and $\mathcal{C} \backslash T$
satisfy (C1) and (C2). Thus, we obtain the new independence systems
$M \backslash T \coloneqq (E \setminus T, \mathcal{C} \backslash T)$ and
$M/T \coloneqq (E \setminus T, \mathcal{C}/T)$ from an independence system $M$.
If $M$ is a matroid, $M \backslash T$ and $M/T$ evidently
correspond to the deletion and contraction of matroids, respectively.
As we make these operations agree with the corresponding operations for matroids,
the following is immediate (see, for example, \cite[Proposition~3.9]{GordonMcNulty12}):

\begin{prop}[{\cite[Proposition~3.9]{GordonMcNulty12}}] \label{prop:single_element_del_contr}
  Let $(E, \mathcal{C})$ be an independence system and
  $e$ an element that is neither an isthmus nor a loop. Then
  \begin{enumerate}[label=\textup{(\arabic*)}, itemsep=0ex]
    \item $D$ is a member of $\mathcal{C} \backslash \{ e \}$
      if and only if $e \notin D$ and $D$ is a member of $\mathcal{C}$.
    \item $D$ is a member of $\mathcal{C}/\{ e \}$ if and only if
      \begin{enumerate}[label=\textup{(\roman*)}, itemsep=0ex]
        \item $D \cup \{ e \}$ is a member of $\mathcal{C}$, or
        \item $D$ is a member of $\mathcal{C}$ and
          $D \cup \{ e \}$ contains no circuits except $D$.
      \end{enumerate}
  \end{enumerate}
\end{prop}

For a matroid $M$, two distinct elements $e, f\in E$ are called
\emph{parallel} if $\{ e, f \}$ is a circuit.
A matroid is called \emph{simple} if it has neither loops nor
parallel elements.
We write $U_{k,n}$ for the \emph{uniform matroid} of rank $k$
on an $n$-element ground set, that is,
the matroid whose independent sets are precisely
the subsets of cardinality at most $k$.

\subsection{Codes over Rings and Independence} \label{subsect:codes_over_rings_and_independence}

Throughout the paper, we shall assume that all rings are commutative
and have the multiplicative identity.
A ring is called \emph{local} if it has a unique maximal ideal.
Unless otherwise noted, we assume that $R$ is a local ring
with maximal ideal $\mathfrak{m}$;
then set $R^{\times} \coloneqq R \setminus \mathfrak{m}$,
the group of units of $R$.
For a subset $S \subseteq V$, we denote $\langle S \rangle_{R}$ or simply $\langle S \rangle$
the $R$-submodule generated by the elements in $S$.
If $S = \{ v_{1}, \dots, v_{\ell} \}$, we also denote $\langle v_{1}, \dots, v_{\ell} \rangle \coloneqq \langle S \rangle$.
For an $R$-module $V$ and an ideal $\mathfrak{a}$ of $R$,
we write $\mathfrak{a}V \coloneqq \{ av \mid a \in \mathfrak{a}, \, v \in V \}$.
By the fundamental fact in algebra, $\mathbb{F} = R/\mathfrak{m}$ is a field.
If $V$ is an $R$-module, then we regard $V/\mathfrak{m}V$ as an $R/\mathfrak{m}$-module,
that is, an $\mathbb{F}$-vector space in the canonical way.
Especially, $V/\mathfrak{m}V$ is finite dimensional
as an $\mathbb{F}$-vector space if $V$ is finitely generated.
For all positive integers $k$, let $R^{k}$ denote the $R$-module of
all ordered $k$-tuples of elements over $R$,
equipped with component-wise addition and scalar multiplication.
Recall that an $R$-module $V$ is \emph{free} if $V \cong R^{k}$ for some integer $k \geq 0$.
For any $R$-module $V$ and nonempty finite set $J$,
we also denote by $V^{J}$ the $R$-module of the functions
from $J$ to $V$. There is a canonical isomorphism
\begin{equation*}
  \begin{array}{ccc}
    V^{J} & \xrightarrow{\sim} & \prod_{j \in J} V \\
    (v \colon J \ni j \mapsto v(j) \in V) &\mapsto& (v(j))_{j \in J}.
  \end{array}
\end{equation*}
Especially, if $V = R$, then under this isomorphism,
we write members of $R^{J}$ in boldface (say, $\bm{v}$);
for $j \in J$ we write $v_{j} \coloneqq \bm{v}(j)$
for the component at $j$, so $\bm{v} = (v_{j})_{j \in J}$.
Similarly, for any two finite sets $J$ and $E$,
we identify the set $R^{J \times E}$ of the mappings from $J \times E$
to $R$, after choosing an ordering of $J$ and $E$, with the set
$R^{\abs{J} \times \abs{E}}$ of $\abs{J} \times \abs{E}$ matrices over $R$.
Under this identification, we write a matrix in $R^{J \times E}$
as, say, $A = (a_{je})_{j \in J, e \in E}$.
Also, for $J^{\prime} \subseteq J$ and $E^{\prime} \subseteq E$,
let $A_{J^{\prime}, E^{\prime}} = (a_{je})_{j \in J^{\prime}, e \in E^{\prime}}$
be the induced submatrix of $A$ by $J^{\prime}$ and $E^{\prime}$.
If $U$ is an $R$-submodule of an $R$-module $V$,
we write $U \leq V$.

For a finitely generated $R$-module $V$, we define
\begin{equation} \label{eq:definition_of_muR}
  \mu_{R}(V) \coloneqq \dim_{R/\mathfrak{m}}(\overline{V}),
\end{equation}
where $\overline{V} \coloneqq V/\mathfrak{m}V$.
The following theorem is a fundamental result in ring theory
derived by Nakayama's lemma (or a theorem of Krull and Azumaya).

\begin{thm}[{\cite[Theorem~2.3]{Matsumura_1987}}] \label{thm:number_of_minimal_generators}
  Let $R$ be a local commutative ring with maximal ideal $\mathfrak{m}$.
  Let $V$ be a finitely generated $R$-module.
  Then, setting $\mathbb{F} \coloneqq R/\mathfrak{m}$ and $\ell \coloneqq \mu_{R}(V)$,
  it follows that
  \begin{enumerate}[label=\textup{(\roman{*})}, itemsep=0ex]
    \item If we take a basis $\{ \overline{\bm{v}}_{1}, \dots, \overline{\bm{v}}_{\ell} \}$
      for $\overline{V}$ over $\mathbb{F}$, and
      choose an inverse image $\bm{v}_{i} \in V$ of
      each $\overline{\bm{v}}_{i}$,
      then $\{ \bm{v}_{1}, \dots, \bm{v}_{\ell} \}$ is a minimal set of
      generators of $V$;
    \item conversely every minimal set of generators of
      $V$ is obtained in this way, and so has $\ell$ elements.
    \item If $\{ \bm{v}_{1}, \dots, \bm{v}_{\ell} \}$
      and $\{ \bm{w}_{1}, \dots, \bm{w}_{\ell} \}$ are
      both minimal sets of generators of $V$,
      and $\bm{w}_{i} = \sum a_{ij} \bm{v}_{j}$ with
      $a_{ij} \in R$ then $\det (a_{ij})$ is a unit of $R$,
      so that $(a_{ij})$ is an invertible matrix.
  \end{enumerate}
\end{thm}

Now we collect the basic ring-theoretic facts needed for the arguments in
Section~\ref{sect:matroid_repr_by_modular_indep}.
In that section we impose no finiteness hypothesis on the ring $R$,
in keeping with the matroid-theoretic practice of considering representations over
both finite and infinite fields (e.g., $\mathbb{R}$ and $\mathbb{C}$).

\begin{dfn}[{\cite{Jilyana99}}] \label{dfn:chain_ring}
  A (commutative) ring $R$ is a \emph{chain ring} if its ideals are
  linearly ordered by inclusion.
\end{dfn}

\begin{dfn}[{\cite{Jilyana99}}] \label{dfn:PIR}
  A ring $R$ is called a \emph{principal ideal ring} (abbreviated PIR)
  if, for any ideal $\mathfrak{a}$ of $R$, there exists $x \in \mathfrak{a}$
  such that $\mathfrak{a} = Rx = xR$. 
\end{dfn}

\begin{lem}[{\cite[Lemma~1.15]{Jilyana99}}] \label{lem:noetherian_chain_iff_local_PIR}
  If $R$ is a local ring with maximal ideal $\mathfrak{m}$,
  which is not necessarily Noetherian but satisfies $\bigcap_{\ell \geq 1} \mathfrak{m}^{\ell} = \{ 0 \}$,
  then the following conditions on $R$ are equivalent:
  \begin{enumerate}[label=\textup{(\arabic*)}, itemsep=0ex]
    \item $\mathfrak{m}$ is principal;
    \item $R$ is a \textup{PIR};
    \item $R$ is a chain ring, hence $R$ is Noetherian.
  \end{enumerate}
\end{lem}

\begin{lem}[{\cite[Lemma~1(ii)]{Guyot17}}] \label{lem:decomposition_PIR}
  Let $R$ be a \textup{PIR}. 
  Then every finitely generated $R$-module $V$ has a unique invariant
  factor decomposition, i.e., a decomposition of the form
  \begin{equation*}
    V \cong R/\mathfrak{a}_{1} \times R/\mathfrak{a}_{2} \times \dots \times R/\mathfrak{a}_{k}
    \quad \text{with} \quad
    R \neq \mathfrak{a}_{1} \supseteq \mathfrak{a}_{2} \supseteq \dots \supseteq \mathfrak{a}_{k}
  \end{equation*}
  where the factors $\mathfrak{a}_{i}$ are uniquely determined by the latter condition.
  For such decomposition, $k$ is the minimal number of generators of $V$, that is $k = \mu_{R}(V)$.
\end{lem}

For an $R$-module $V$ and $\bm{v} \in V$, the ideals
\begin{align*}
  \Ann_{R}(\bm{v}) &\coloneqq \{ r \in R \mid r\bm{v} = \bm{0} \}, \\
  \Ann_{R}(V) &\coloneqq \{ r \in R \mid rV = \{ \bm{0} \} \} = \bigcap_{\bm{v} \in V} \Ann_{R}(\bm{v})
\end{align*}
are called the \emph{annihilators} of $\bm{v}$ and $V$, respectively.
Then we define
\begin{align*}
  \psi(V) &\coloneqq \{ \bm{v} \in V \mid \Ann_{R}(\bm{v}) = \{ 0 \} \} \\
  &= \{ \bm{v} \in V \setminus \{ \bm{0} \} \mid \alpha \bm{v} \neq \bm{0} \, \text{ for all } \alpha \in \mathfrak{m} \setminus \{ 0 \} \}.
\end{align*}
(The exclusion of $\bm{0}$ on the second line is only to keep the
equivalence with the first line when $R$ is a field: in that case
$\mathfrak m = \{ 0 \}$ and $\mathfrak m\setminus\{0\}=\emptyset$, so the
universal condition is vacuous. Otherwise, $\bm{0}$ is excluded
automatically.)
It is easy to check that $\psi(V)$ is stable under multiplication
by units: that is, for every $u \in R^{\times}$ and $\bm{v} \in \psi(V)$,
we have $u\bm{v} \in \psi(V)$.
Note that it is not necessarily the case that
$\bm{v} + \bm{v}^{\prime} \in \psi(V)$ for $\bm{v}, \bm{v}^{\prime} \in \psi(V)$;
therefore $\psi(V)$ need not be an $R$-submodule of $V$.

\begin{lem} \label{lem:psi_and_Ann}
  For a commutative local ring $(R, \mathfrak{m})$,
  the following are equivalent:
  \begin{enumerate}[label=\textup{(\roman*)}, itemsep=0ex]
    \item For any finite set $E$ and every $R$-submodule $C \leq R^{E}$,
      \begin{equation*}
        \psi(C) = \{ (c_{i})_{i \in E} \in C \mid \exists i \in E, \, c_{i} \in R^{\times} \}.
      \end{equation*}
    \item For all finitely generated ideal $\mathfrak{a} \subseteq \mathfrak{m}$,
      \begin{equation*}
        \Ann_{R}(\mathfrak{a}) \neq \{ 0 \}.
      \end{equation*}
  \end{enumerate}
\end{lem}
\begin{proof}
  (i) $\Rightarrow$ (ii):
  Let $\mathfrak{a} = \langle a_{1}, \dots, a_{n} \rangle \subseteq \mathfrak{m}$
  be a finitely generated ideal of $R$.
  Set $E \coloneqq \{ 1, \dots, n \}$ and $\bm{c} \coloneqq (a_{1}, \dots, a_{n}) \in R^{E}$.
  For all $i \in E$, we have $a_{i} \in \mathfrak{m}$;
  thus applying (i) to the $R$-submodule $C = R^{E}$, we obtain $\bm{c} \notin \psi(R^{E})$.
  Moreover, by the definition of $\psi$, we get $\Ann_{R}(\bm{c}) \neq \{ 0 \}$.
  Since $\Ann_{R}(\bm{c}) \subseteq \Ann_{R}(\mathfrak{a})$, we have $\Ann_{R}(\mathfrak{a}) \neq \{ 0 \}$.

  (ii) $\Rightarrow$ (i):
  Set $S \coloneqq \{ (c_{i})_{i \in E} \in C \mid \exists i \in E, \, c_{i} \in R^{\times} \}$.
  Fix a vector $\bm{c} = (c_{i})_{i \in E} \in C$.
  If there exists $i \in E$ with $c_{i} \in R^{\times}$, then
  $\psi(C) \supseteq S$ since $\Ann_{R}(\bm{c}) = \{ 0 \}$.
  Assume that $c_{i} \in \mathfrak{m}$ for all $i \in E$, and
  set $\mathfrak{a} \coloneqq \langle \{ c_{i} \}_{i \in E} \rangle \subseteq \mathfrak{m}$.
  Since $\mathfrak{a}$ is finitely generated, the assumption (ii)
  guarantees the existence of a non-zero $\alpha \in \Ann_{R}(\mathfrak{a})$ such that $\alpha \bm{c} = \bm{0}$.
  Hence, $\bm{c} \notin \psi(C)$, so $\psi(C) \subseteq S$.
  Therefore, $\psi(C) = S$.
\end{proof}

\begin{rem} \label{rem:psi_and_Ann}
  If $R$ is finite (or Artinian in general), (ii) always holds, and so (i) does.
  Indeed, if $R$ is Artinian, the maximal ideal $\mathfrak{m}$ is nilpotent.
  Let $\nu$ be the nilpotency index of $\mathfrak{m}$ and
  $\mathfrak{a} \subseteq \mathfrak{m}$ be a finitely generated ideal.
  Then for any $x \in \mathfrak{m}^{\nu - 1}$ and any $a \in \mathfrak{a} \subseteq \mathfrak{m}$,
  we have
  \begin{equation*}
    xa \in \mathfrak{m}^{\nu - 1} \mathfrak{m} = \mathfrak{m}^{\nu} = 0.
  \end{equation*}
  Hence $\mathfrak{m}^{\nu - 1} \subseteq \Ann_{R}(\mathfrak{a})$.
  Since $\mathfrak{m}^{\nu - 1} \neq \{ 0 \}$ by the minimality of $\nu$,
  it follows that $\Ann_{R}(\mathfrak{a}) \neq 0$, and therefore (ii) holds.

  We next give a counterexample in the non-Artinian case.
  Let $\mathbb{F}$ be a field, and consider the formal power series ring
  $R = \mathbb{F} \llbracket t \rrbracket$ in one variable $t$.
  This is a commutative local ring with maximal ideal $\mathfrak{m} = \langle t \rangle$,
  but it does not satisfy (ii).
  Indeed, take the principal ideal $\mathfrak{a} \coloneqq \langle t \rangle = \mathfrak{m}$.
  Since $R$ is an integral domain (see \cite[Example 1 in Section 1]{Matsumura_1987}),
  we have
  \begin{equation*}
    \Ann_{R}(\mathfrak{a}) = \Ann_{R}(t) = \{ a \in R \mid at = 0 \} = 0.
  \end{equation*}
  Thus (ii) fails. In this situation, let $E = \{ 1 \}$ and $C = R \leq R^{E}$.
  Then $\bm{c} = (t) \in C$, and the computation above shows that
  $\bm{c} \in \psi(C)$. However, since $t$ is not a unit, (i) does not hold.
\end{rem}

\begin{dfn}\label{dfn:primitive_element}
  An element $\bm{v} \in V$ is called \emph{primitive}
  if $\bm{v} \notin \mathfrak{m}V$.
  Equivalently, its image in $V/\mathfrak{m}V$ is nonzero.
\end{dfn}

For two vectors $\bm{x} = (x_{i})_{i \in E}, \, \bm{y} = (y_{i})_{i \in E} \in R^{E}$ , define the \emph{inner product} as
\begin{equation*}
  \bm{x} \cdot \bm{y} = \sum_{i \in E} x_{i} y_{i}.
\end{equation*}

A \emph{code} $C$ of \emph{length} $n$ over $R$ is a nonempty subset of $R^{E}$,
where $n \coloneqq \abs{E}$. Each vector in $C$ is called a \emph{codeword},
and each element in $E$ is called a \emph{coordinate}.
The code $C$ is \emph{linear} (or an \emph{$R$-code}) if it is an $R$-submodule of $R^{E}$.
The \emph{dual code} $C^{\perp}$ of $C$ is defined as
\begin{equation*}
  C^{\perp} \coloneqq \{ \bm{x} \in R^{E} \mid \bm{x} \cdot \bm{c} = 0 \text{ for all } \bm{c} \in C \}.
\end{equation*}
Note that $C^{\perp}$ is also an $R$-submodule of $R^{E}$.
A matrix $G$ whose rows form a minimal generating set for $C$
is called a \emph{generator matrix} for $C$.
A generator matrix $H$ for $C^{\perp}$ is called a \emph{parity-check matrix} for $C$.

We record the following standard systematic form for free codes
over finite commutative local rings, since it will be used repeatedly
later; see, for example, \cite[Section~3.2]{SystematicForm}.
To describe the code itself, one uses generator matrices
up to elementary row operations and coordinate permutations.

\begin{prop}[{\cite{SystematicForm}}] \label{prop:systematic_form_for_free_codes}
  Let $R$ be a finite commutative local ring, and
  let $C \leq R^{E}$ be a free $R$-code of rank $k \coloneqq \mu_{R}(C)$.
  Then, after a permutation of the coordinates and elementary row operations,
  a generator matrix of $C$ can be written in the form
  \begin{equation*}
    G = \lbrack I_{k} \mid A \rbrack.
  \end{equation*}
  In this case, a parity-check matrix for $C$ is given by
  \begin{equation*}
    H = \lbrack -A^{\top} \mid I_{\abs{E}-k} \rbrack,
  \end{equation*}
  and in particular $C^{\perp}$ is free of rank $\abs{E}-k$.
  Consequently,
  \begin{equation*}
    C = \{ (\bm{u}^{\top}, \bm{u}^{\top}A) \mid \bm{u} \in R^{k} \}, \quad
    C^{\perp} = \{ (-\bm{v}^{\top}A^{\top}, \bm{v}^{\top}) \mid \bm{v} \in R^{\abs{E} - k} \}.
  \end{equation*}
\end{prop}

\begin{rem} \label{rem:counterexample_of_systematic_form_in_nonlocal_case}
  The systematic-form description above may fail if $R$ is not local.
  For example, consider $C = \langle (2, 3) \rangle \leq \mathbb{Z}_{6}^{2}$.
  Then the mapping $\mathbb{Z}_{6} \to \mathbb{Z}_{6}^{2}, \, r \mapsto r(2, 3)$ is injective,
  so $C \cong \mathbb{Z}_{6}$. However, it is impossible to transform
  a generator matrix $(2 \; 3)$ into $(1 \; a)$ or $(a \; 1)$
  only by the elementary row operations and coordinate permutations.
  For more details on the discussion for a finite commutative local ring,
  see \cite[Section~3.2]{SystematicForm}.
\end{rem}

The \emph{support} of each vector $\bm{x} = (x_{i})_{i \in E} \in R^{E}$
is defined as:
\begin{equation*}
  \supp(\bm{x}) \coloneqq \{ i \in E \mid x_{i} \neq 0 \}.
\end{equation*}
For each subset $X \subseteq E$ and an $R$-code $C$,
define the \emph{punctured code} $C^{X}$ to be the $R$-submodule of $R^{E \setminus X}$,
obtained by deleting the coordinates $X$ from each codeword of $C$.
Also, define the \emph{shortened code} $C_{X}$ to be the $R$-submodule of $R^{E \setminus X}$
obtained by deleting the coordinates $X$ from each codeword $\bm{c} \in C$
with $\supp(\bm{c}) \cap X = \emptyset$. Formally,
\begin{align*}
  C^{X} &\coloneqq \{ \bm{c}|_{E \setminus X} : \bm{c} \in C \} \leq R^{E \setminus X}; \\
  C_{X} &\coloneqq \{ \bm{c}|_{E \setminus X} : \bm{c} \in C, \, \supp(\bm{c}) \cap X = \emptyset \} \leq R^{E \setminus X},
\end{align*}
respectively, where $f|_{S}$ is the restriction of $f \colon E \to R$ to $S \subseteq E$.

\begin{lem}[{\cite[Lemma~1]{DemiMatroid}}] \label{lem:duality_for_puncturing_and_shortening}
  For any $R$-code $C$ and any subset $X \subseteq E$,
  \begin{equation*}
    (C_{X})^{\perp} = (C^{\perp})^{X}
    \quad \text{and} \quad
    (C^{X})^{\perp} = (C^{\perp})_{X}.
  \end{equation*}
\end{lem}

Because finite commutative chain rings will play a central role
throughout this paper, we briefly summarise the required background.
Since $R$ is finite, it is Artinian, and hence $\bigcap_{\ell \geq 1} \mathfrak{m}^{\ell} = \{ 0 \}$.
Therefore, by Lemma~\ref{lem:noetherian_chain_iff_local_PIR},
$R$ is a principal ideal ring; in particular, its maximal ideal
$\mathfrak{m}$ is principal. 
Fix a generator $\theta$ of $\mathfrak{m}$.
Since $R$ is a local Artinian ring, the maximal ideal $\mathfrak{m}$
is nilpotent (cf. \cite[Corollary~8.2 and Proposition~8.4]{AtiyahMacDonald}).
Let $\nu$ denote the \emph{nilpotency index} of $\mathfrak{m}$,
that is, the smallest integer such that $\mathfrak{m}^{\nu} = \{ 0 \}$.
Then the ideals of $R$ are exactly
\begin{equation*}
  R = \langle \theta^{0} \rangle \supsetneq
  \langle \theta^{1} \rangle \supsetneq \dots
  \supsetneq \langle \theta^{\nu - 1} \rangle \supsetneq
  \langle \theta^{\nu} \rangle = \{ 0 \}.
\end{equation*}

Let $q \coloneqq \abs{R/\mathfrak{m}}$.
Since $\mathfrak{m} = \langle \theta \rangle$,
for each $0 \leq i \leq \nu - 1$ multiplication by $\theta^{i}$
induces an isomorphism
\begin{equation*}
  R/\langle \theta \rangle \xrightarrow{\sim} \langle \theta^{i} \rangle/\langle \theta^{i+1} \rangle,
  \qquad
  \overline{r} \longmapsto r\theta^i + \langle \theta^{i + 1} \rangle.
\end{equation*}
Thus every successive quotient $\langle \theta^{i} \rangle/\langle \theta^{i+1} \rangle$
has cardinality $q$. It follows that
$\abs{\mathfrak{m}^{i}} = \abs{\langle \theta^{i} \rangle} = q^{\nu - i}$ for all $0 \leq i \leq \nu$,
and in particular, $\abs{R} = q^{\nu}$ and $\abs{\mathfrak{m}} = q^{\nu - 1}$.

The same chain of ideals also gives a valuation-like description of elements.
If $a \in R \setminus \{ 0 \}$, then the principal ideal
$\langle a \rangle$ must be one of the ideals $\mathfrak m^{t} = \langle \theta^{t} \rangle$,
say $\langle a \rangle = \langle \theta^{t} \rangle$, for a unique $t \in \{0, \dots, \nu-1\}$.
Writing $a = u\theta^{t}$, the coefficient $u$ cannot lie in $\langle \theta \rangle$;
for otherwise $a \in \langle \theta^{t+1} \rangle$, contradicting
$\langle a \rangle = \langle \theta^{t} \rangle$.
Hence, for every $a \in R \setminus \{ 0 \}$, there is an expression
$a = u\theta^{t}$ using a unit $u \in R^{\times}$ and a unique exponent $t$.
Define a function $\ord_{\theta} \colon R \to \{ 0, \dots, \nu \}$ as
\begin{equation*}
  \ord_{\theta}(a)
  \coloneqq
  \max\{ t \in \{ 0, \dots, \nu \} : a \in \langle \theta^{t} \rangle\}.
\end{equation*}
Then,
\begin{equation*}
  \abs{\{a \in R : \ord_{\theta}(a) = t \}}
  = \abs{\langle \theta^{t} \rangle} - \abs{\langle \theta^{t+1} \rangle}
  = q^{\nu-t-1}(q-1)
\end{equation*}
for $0 \leq t \leq \nu-1$.

This description of elements immediately identifies the cyclic $R$-modules.
Now let $\bm{v} \in R^{E}$ be non-zero. The map
$\varphi_{\bm{v}} \colon R \to \langle \bm{v} \rangle, \; a \mapsto a\bm{v}$,
is surjective, and its kernel is $\Ann_{R}(\bm{v})$.
Hence, by the first isomorphism theorem, $\langle \bm{v} \rangle \cong R/\Ann_{R}(\bm{v})$.
Since every ideal of $R$ is a power of $\langle \theta \rangle$,
every non-zero cyclic $R$-module is isomorphic to $R/\langle \theta^{s} \rangle$
for a unique $s \in \{1, \dots, \nu\}$, or equivalently
to $\langle \theta^{\nu-s} \rangle$.
Moreover, a cyclic module of type $R/\langle \theta \rangle^{s}$ has exactly
$q^{s-1}(q-1)$ generators, because its non-generators are precisely
the elements of its unique maximal submodule.

These cyclic modules are the building blocks of all finitely generated
$R$-modules. Since $R$ is a PIR by Lemma~\ref{lem:noetherian_chain_iff_local_PIR},
every finitely generated $R$-module admits an invariant factor decomposition.
Equivalently, every matrix $A \in R^{k \times n}$ admits a \emph{Smith normal form}
(abbreviated SNF): there exist invertible matrices $P \in R^{k \times k}$
and $Q \in R^{n \times n}$ such that
\begin{equation*}
  PAQ = \diag(\theta^{\lambda_{1}}, \dots, \theta^{\lambda_{r}}, 0, \dots, 0)
\end{equation*}
with
\begin{equation*}
  0 \le \lambda_1 \le \dots \le \lambda_r \le \nu-1.
\end{equation*}
See \cite[Chapter~15]{Brown93} for more details on SNF over commutative PIRs.
Applied to a generator matrix of a code $C \leq R^{E}$,
this determines the isomorphism type of $C$ as an $R$-module.
However, Smith normal form uses column operations and therefore
does not preserve the embedded code $C \subseteq R^{E}$.
For coding-theoretic purposes one therefore uses a standard form
generator matrix, obtained by row operations together with
a permutation of the coordinates.

For later enumeration, it is convenient to record the following.
For a finite $R$-module $U$ and $s \in \{ 0, \dots, \nu \}$, define
\begin{equation*}
  U[\theta^{s}] \coloneqq \{ \bm{u} \in U \mid \theta^{s} \bm{u} = \bm{0} \}.
\end{equation*}
If
\begin{equation*}
  U \cong R^{k_0} \oplus \bigoplus_{i=1}^{\nu-1}
  \bigl( R/\langle \theta^{\nu-i} \rangle \bigr)^{k_{i}},
\end{equation*}
then each summand $R/\langle \theta^{\nu-i} \rangle$ contributes
$q^{\min(s,\nu-i)}$ elements to $U\lbrack \theta^{s} \rbrack$.
Indeed, for a summand $R/\langle \theta^{r} \rangle$,
the elements killed by $\theta^{s}$ form the whole summand if $s \geq r$,
and otherwise the submodule $\langle \theta^{r-s} \rangle/\langle \theta^{r} \rangle$.
Hence there are exactly $q^{\min(s,r)}$ such elements.
Therefore,
\begin{equation*}
  \abs{U\lbrack \theta^{s} \rbrack}
  = q^{\sum_{i=0}^{\nu-1} k_{i} \min(s, \nu-i)}.
\end{equation*}
Hence we have the following counting formula:
\begin{prop} \label{prop:counting_formula}
  The number $N_{U}(s)$ of cyclic submodules of $U$ isomorphic to
  $R/\langle \theta^{s} \rangle$ is
  \begin{equation*}
    N_{U}(s)
    =
    \frac{\abs{U\lbrack \theta^{s} \rbrack} - \abs{U\lbrack \theta^{s-1} \rbrack}}{q^{s-1}(q-1)}
    \qquad (1 \le s \le \nu).
  \end{equation*}
\end{prop}
This counting formula is a special case of the enumeration of
submodules of prescribed shape due to \cite[Theorem~2.4]{HL00}.
Intuitively, the numerator counts the elements annihilated by $\theta^{s}$
but not by $\theta^{s-1}$, while the denominator is the
number of generators of a cyclic module of type $R/\langle \theta^{s} \rangle$.

We isolate some special cases needed later for representations.

\begin{cor}\label{cor:counting_formula}
  The number of distinct cyclic submodules $\langle \bm{v} \rangle \leq R^{k}$
  generated by primitive vectors in $R^{k}$ is
  \begin{equation*}
    \frac{q^{\nu k}-q^{(\nu-1)k}}{q^\nu-q^{\nu-1}}
    =
    q^{(\nu-1)(k-1)}\frac{q^k-1}{q-1}.
  \end{equation*}
  In particular, the number of cyclic submodules $\langle \bm{v} \rangle \leq R^{2}$
  with $\bm{v} \notin \theta R^{2}$ is
  \begin{equation*}
    q^{\nu} + q^{\nu-1} = \abs{R} + \abs{\mathfrak{m}}.
  \end{equation*}
\end{cor}

\begin{proof}
  The number of primitive vectors in $R^{k}$ is
  $\abs{R}^{k} - \abs{\mathfrak{m}}^{k} = q^{\nu k}-q^{(\nu - 1)k}$.
  Two primitive vectors in $R^{k}$ generate the same cyclic submodule
  iff they differ by a unit scalar,
  and each such submodule has exactly
  $\abs{R^{\times}} = \abs{R} - \abs{\mathfrak{m}} = q^{\nu} - q^{\nu-1}$
  primitive generators. Dividing yields the former claim.
  The latter claim is obtained by substituting $k = 2$.
\end{proof}

For more details on the structure of linear codes over finite commutative chain rings,
see \cite{HL00, HL09, NS00}.

We now return to general local commutative rings.
Modular independence is originally defined by Park \cite{Park09},
and generalised to the case of finite commutative Frobenius rings
by Dougherty and Liu \cite{DoughertyLiu09}.
In this paper, we define modular independence under the assumption of
locality, but without finiteness. We conclude this section with some
properties of modular independence over local commutative rings.

\begin{dfn}[{\cite[Definition~1]{DoughertyLiu09}}] \label{dfn:modular_independence}
  Let $R$ be a local commutative ring with unique maximal ideal
  $\mathfrak{m}$. Let $V$ be an $R$-module.
  $\bm{v}_{1}, \bm{v}_{2}, \dots, \bm{v}_{\ell} \in V$ are
  \emph{modular independent} if
  $\alpha_{1} \bm{v}_{1} + \dots + \alpha_{\ell} \bm{v}_{\ell} = \bm{0}$
  implies $\alpha_{i}$ is a non-unit, that is,
  $\alpha_{i} \in \mathfrak{m}$, for all $i = 1, 2, \dots, \ell$;
  otherwise they are said to be \emph{modular dependent}.
\end{dfn}

\begin{rem} \label{rem:modular_indep_with_syzygy}
  In other words, $\bm{v}_{1}, \dots, \bm{v}_{\ell} \in V$ are
  modular independent if the first syzygy of
  $V^{\prime} \coloneqq \langle \bm{v}_{1}, \dots, \bm{v}_{\ell} \rangle_{R}$
  with respect to $\bm{v}_{1}, \dots, \bm{v}_{\ell}$ defined as
  the kernel of the surjection
  \begin{equation*}
    \varphi_{\bm{v}_{1}, \dots, \bm{v}_{\ell}} \colon R^{\ell} \to V^{\prime}, \;
    (\alpha_{1}, \dots, \alpha_{\ell}) \mapsto \sum_{i = 1}^{\ell} \alpha_{i} \bm{v}_{i}
  \end{equation*}
  is included in $\mathfrak{m}R^{\ell}$ , that is,
  \begin{equation*}
\Ker \varphi_{\bm{v}_{1}, \dots, \bm{v}_{\ell}}
    =  \left\{ (\alpha_{1}, \dots, \alpha_{\ell}) \in R^{\ell} \mathrel{}\middle|\mathrel{} \sum_{i = 1}^{\ell} \alpha_{i} \bm{v}_{i} = \bm{0} \right\} \subseteq \mathfrak{m}R^{\ell}.
  \end{equation*}
\end{rem}

As in the case of linear independence,
modular dependence is characterised in terms of linear combinations.

\begin{lem}[{\cite[Lemma~3.1]{DoughertyLiu09}}] \label{lem:modular_dep_lin_comb}
  Let $V$ be an $R$-module.
  Then $\bm{v}_{1}, \dots, \bm{v}_{\ell} \in V$ are modular dependent
  if and only if some $\bm{v}_{j}$ is a linear combination of
  the other vectors. 
\end{lem}

For a pair of vectors, Lemma~\ref{lem:modular_dep_lin_comb} specialises as follows.

\begin{cor}\label{cor:two_vectors_mod_dep}
  Let $R$ be a local commutative ring,
  and let $\bm{v}, \bm{w}$ be non-zero vectors in an $R$-module $V$.
  Then $\{ \bm{v}, \bm{w} \}$ is modular dependent if and only if
  $\bm{v} \in \langle \bm{w} \rangle$ or $\bm{w} \in \langle \bm{v} \rangle$.
\end{cor}

\begin{proof}
  If $\{ \bm{v}, \bm{w}\}$ is modular dependent, then Lemma~\ref{lem:modular_dep_lin_comb}
  implies that one of $\bm{v}, \bm{w}$ is a linear combination of the other.
  The converse is immediate from Definition~\ref{dfn:modular_independence}.
\end{proof}

Using the determinant, we give a sufficient condition for vectors to be modular independent.

\begin{prop} \label{prop:nonzero_minor_implies_modular_indep}
  Let $\bm{v}_{1}, \dots, \bm{v}_{\ell} \in R^{k}$ where $\ell \leq k$.
  Let $A$ be the $k \times \ell$ matrix whose columns are
  $\bm{v}_{1}, \dots, \bm{v}_{\ell}$.
  If there exists a non-zero $\ell \times \ell$ minor of $A$,
  then $\bm{v}_{1}, \dots, \bm{v}_{\ell}$ are modular independent.
\end{prop}

\begin{proof}
  Let $A^{\prime}$ be an $\ell \times \ell$ submatrix of $A$ such that
  $\det A^{\prime} \neq 0$. Take the cofactor matrix $\widetilde{A^{\prime}}$
  of $A^{\prime}$, then
  $\widetilde{A^{\prime}}A^{\prime} = (\det A^{\prime})I_{\ell}$,
  where $I_{\ell}$ is the $\ell$-th identity matrix over $R$.
  Define $\varphi_{A^{\prime}} \colon R^{\ell} \to R^{\ell}$
  as $\varphi_{A^{\prime}}(\bm{a}) = A^{\prime}\bm{a}$ for all $\bm{a} \in R^{\ell}$.
  If $\bm{a} \in \Ker \varphi_{A^{\prime}}$, then
  \begin{equation*}
    (\det A^{\prime}) \bm{a}
    = (\det A^{\prime}) I_{\ell} \bm{a}
    = \widetilde{A^{\prime}}A^{\prime} \bm{a}
    = \widetilde{A^{\prime}} \bm{0}
    = \bm{0}.
  \end{equation*}
  The linear equation $(\det A^{\prime}) a_{i} = 0$ for each coordinate of $\bm{a} = (a_{1}, \dots, a_{\ell})^{\top}$
  forces $a_{i} \in \mathfrak{m}$ for all $i \in \{ 1, \dots, \ell \}$
  because $\det A^{\prime} \neq 0$.
  Thus, we have $\bm{a} \in \mathfrak{m} R^{\ell}$.
  Therefore, the column vectors in $A^{\prime}$ are modular independent,
  and so are those in $A$.
\end{proof}

For all integers $m \geq 2$,
let $\mathbb{Z}_{m} \coloneqq \mathbb{Z}/m\mathbb{Z}$ denote the ring of integers modulo $m$.

\begin{eg}
  Since
  $\det \left( \begin{smallmatrix} 1 & 1 \\ 0 & 2 \end{smallmatrix} \right) = 2 \neq 0$
  over $\mathbb{Z}_{4}$, $(1, 0)^{\top}, (1, 2)^{\top} \in \mathbb{Z}_{4}^{2}$
  are modular independent.
  However, the converse of Proposition~\ref{prop:nonzero_minor_implies_modular_indep} does not necessarily hold.
  Indeed, $(2, 0)^{\top}, (0, 2)^{\top} \in \mathbb{Z}_{4}^{2}$
  are modular independent although
  $\det \left( \begin{smallmatrix} 2 & 0 \\ 0 & 2 \end{smallmatrix} \right) = 0$
  over $\mathbb{Z}_{4}$.
\end{eg}
 
\section{Representation of Matroids by Modular Independence} \label{sect:matroid_repr_by_modular_indep}

In this section, we generalise the construction of vector matroids
over a field to a local commutative ring by using modular independence.
No finiteness assumption on the ring $R$ will be imposed
in this section, in line with the matroid-theoretic convention
of treating representations over infinite fields
(such as $\mathbb{R}$ and $\mathbb{C}$) on the same footing as
those over finite fields.
One of the fundamental ways to construct a matroid is as follows
(see \cite{Oxley} and \cite{Welsh}):

\begin{prop} \label{prop:vector_matroid}
  Let $A \in V^{E}$ be a mapping from a finite set $E$
  to a vector space $V$ over $\mathbb{F}$.
  If $\mathcal{I}$ is the collection of subsets $I \subseteq E$
  such that $(A(i))_{i \in I}$ is linearly independent over $\mathbb{F}$,
  then $(E, \mathcal{I})$ is a matroid.
\end{prop}

When $\lbrack k \rbrack \coloneqq \{ 1, \dots, k \}$ and
$V = \mathbb{F}^{[k]}$, we identify $A \in (\mathbb{F}^{[k]})^{E}$
with the $k \times \abs{E}$ matrix $A$ over $\mathbb{F}$
whose columns are indexed by $E$.
We denote by $A(S)$ the image of a subset $S \subseteq E$ under $A$,
and distinguish it from $A_{S} \coloneqq (A(i))_{i \in S}$.
Note that multiple members in $A_{S}$ are identified in $A(S)$.
Especially, we may regard $A = A_{E}$.
Then the matroid obtained from the matrix $A$ is called
the \emph{vector matroid} of $A$, and denoted by $M[A]$.
Now we extend the construction of vector matroids
to the case of matrices over a local commutative ring $R$
by using modular independence in place of linear independence.
The routine proof of the following proposition is omitted.

\begin{prop} \label{prop:indep_sys_from_modular_indep}
  Let $A \in V^{E}$ be a mapping from a finite set $E$
  to an $R$-module $V$ over a local commutative ring $R$.
  If $\mathcal{I}$ is the collection of subsets $I \subseteq E$
  such that $(A(i))_{i \in I}$ is modular independent over $R$,
  then $(E, \mathcal{I})$ is an independence system.
\end{prop}

When $R$ is a field, the pair $(E, \mathcal{I})$ obtained
from Proposition~\ref{prop:indep_sys_from_modular_indep} agrees with
the matroid given by Proposition~\ref{prop:vector_matroid}
since the field $\mathbb{F}$ is a local commutative ring
whose unique maximal ideal is $\mathfrak{m} = \{ 0 \}$
and an $\mathbb{F}$-module is a vector space over $\mathbb{F}$.
If $V = R^{\lbrack k \rbrack}$, we identify $A \in (R^{\lbrack k \rbrack})^{E}$
with the $k \times \abs{E}$ matrix $A$ over $R$
whose columns are indexed by $E$.
We denote the independence system obtained from such matrix $A$
by $M\lbrack A \rbrack$.
The pair $(E, \mathcal{I})$ need not be a matroid
as the following example shows.

\begin{eg} \label{eg:non_matroid}
  Set $E \coloneqq \{ 1, 2, 3 \}$.
  The ring $\mathbb{Z}_{4}$ is a local commutative ring with unique maximal ideal
  $\mathfrak{m} = 2\mathbb{Z}_{4} \eqqcolon \langle 2 \rangle$.
  \begin{equation*}
    A = \bordermatrix{
      & 1 & 2 & 3 \cr
      & 2 & 1 & 1 \cr
      & 0 & 0 & 2
    } \in (\mathbb{Z}_{4}^{2})^{E}.
  \end{equation*}
  Clearly, every single column vector in $A$ is modular independent,
  and $\{ 1, 2 \}, \{ 1, 3 \} \notin \mathcal{I}$ since
  $(2, 0)^{\top} = 2(1, 0)^{\top} = 2(1, 2)^{\top}$ over $\mathbb{Z}_{4}$.
  On the other hand, we have $\{ 2, 3 \} \in \mathcal{I}$
  because the linear relation
  \begin{equation*}
    \alpha_{1} \begin{pmatrix} 1 \\ 0 \end{pmatrix} + \alpha_{2} \begin{pmatrix} 1 \\ 2 \end{pmatrix} = \begin{pmatrix} 0 \\ 0 \end{pmatrix}
    \quad (\alpha_{1}, \alpha_{2} \in \mathbb{Z}_{4})
  \end{equation*}
  implies $(\alpha_{1}, \alpha_{2}) \in 2\mathbb{Z}_{4}^{2}$.
  Hence we have
  \begin{equation*}
    \mathcal{I} = \{ \emptyset, \{ 1 \}, \{ 2 \}, \{ 3 \}, \{ 2, 3 \} \},
  \end{equation*}
  but $(E, \mathcal{I})$ is not a matroid.
\end{eg}

We begin by characterising modular independence in terms of $\mu_{R}$ as follows.

\begin{lem} \label{lem:modular_indep_minimal_gen}
  Let $V$ be an $R$-module.
  Then $\bm{v}_{1}, \dots, \bm{v}_{\ell} \in V$ are modular independent
  if and only if
  \begin{equation*}
    \mu_{R}(\langle \bm{v}_{1}, \dots, \bm{v}_{\ell} \rangle_{R}) = \ell.
  \end{equation*}
\end{lem}

\begin{proof}
  If $\bm{v}_{1}, \dots, \bm{v}_{\ell}$ are modular independent,
  no $\bm{v}_{i}$ is a linear combination of vectors
  in $\bm{v}_{1}, \dots, \bm{v}_{i-1}, \bm{v}_{i+1}, \dots, \bm{v}_{\ell}$.
  Thus, $\bm{v}_{1}, \dots, \bm{v}_{\ell}$ is a minimal set of
  generators of $\langle \bm{v}_{1}, \dots, \bm{v}_{\ell} \rangle_{R}$.
  By Theorem~\ref{thm:number_of_minimal_generators},
  we have $\mu_{R}(\langle \bm{v}_{1}, \dots, \bm{v}_{\ell} \rangle_{R}) = \ell$.

  Conversely, assume that
  $\mu_{R}(\langle \bm{v}_{1}, \dots, \bm{v}_{\ell} \rangle_{R}) = \ell$
  and that some $\bm{v}_{i}$ is a linear combination of
  $\bm{v}_{1}, \dots, \bm{v}_{\ell}$.
  Then $\langle \bm{v}_{1}, \dots, \bm{v}_{\ell} \rangle_{R}$ is
  generated by $\ell - 1 = \mu_{R}(\langle \bm{v}_{1}, \dots, \bm{v}_{\ell} \rangle_{R}) - 1$ elements;
  a contradiction to Theorem~\ref{thm:number_of_minimal_generators}.
\end{proof}

Thus, a minimal set of generators of a finitely generated $R$-module $V$
forms a maximal set of modular independent vectors.
However, the converse is not true in general.
That is, a maximal set of modular independent vectors in $V$
does not necessarily form a minimal set of generators of $V$.
For instance, $(1, 0), (1,2) \in \mathbb{Z}_{4}^{2}$ are
maximal modular independent vectors, but they do not generate
$\mathbb{Z}_{4}^{2}$; they generate only $\mathbb{Z}_{4} \oplus 2\mathbb{Z}_{4}$. 

The following lemma is a fundamental fact following from Nakayama's lemma
and the right-exactness of the tensor product.

\begin{lem} \label{lem:muR_supermodular}
  Let $R$ be a local commutative ring with maximal ideal $\mathfrak{m}$,
  $V$ be an $R$-module, $\mathbb{F} \coloneqq R/\mathfrak{m}$ be the residue field.
  Then, $\mu_{R}$ defined in \eqref{eq:definition_of_muR} is supermodular, that is,
  \begin{equation*}
    \mu_{R}(V_{1}) + \mu_{R}(V_{2}) \leq \mu_{R}(V_{1} + V_{2}) + \mu_{R}(V_{1} \cap V_{2})
  \end{equation*}
  for all finitely generated $R$-submodules $V_{1}, V_{2} \leq V$.
  Furthermore, the equality holds if and only if
  \begin{equation*}
    \mathfrak{m}V_{1} \cap \mathfrak{m}V_{2} = \mathfrak{m}(V_{1} \cap V_{2}).
  \end{equation*}
\end{lem}

\begin{proof}
  Set $\iota \colon V_{1} \cap V_{2} \to V_{1} \oplus V_{2}, \, \bm{v} \mapsto (\bm{v}, -\bm{v})$
  and $\pi \colon V_{1} \oplus V_{2} \to V_{1} + V_{2}, \, (\bm{v}_{1}, \bm{v}_{2}) \mapsto \bm{v}_{1} + \bm{v}_{2}$.
  Applying $- \otimes_{R} \mathbb{F}$ to the exact sequence
  $\{ \bm{0} \} \to V_{1} \cap V_{2} \xrightarrow{\iota} V_{1} \oplus V_{2} \xrightarrow{\pi} V_{1} + V_{2} \to \{ \bm{0} \}$,
  we obtain the right-exact sequence, with the induced map $\overline{\iota}$ on the left:
  \begin{equation*}
    (V_{1} \cap V_{2})/\mathfrak{m}(V_{1} \cap V_{2})
    \xrightarrow{\overline{\iota}}
    (V_{1}/\mathfrak{m}V_{1}) \oplus (V_{2}/\mathfrak{m}V_{2})
    \xrightarrow{\overline{\pi}}
    (V_{1} + V_{2})/\mathfrak{m}(V_{1} + V_{2})
    \to
    \{ \bm{0} \}.
  \end{equation*}
  Hence,
  \begin{align*}
    \mu_{R}(V_{1}) + \mu_{R}(V_{2})
    &= \dim_{\mathbb{F}} (V_{1}/\mathfrak{m}V_{1}) + \dim_{\mathbb{F}}(V_{2}/\mathfrak{m}V_{2}) \\
    &= \dim_{\mathbb{F}} \Im(\overline{\iota}) + \dim_{\mathbb{F}}((V_{1} + V_{2})/\mathfrak{m}(V_{1} + V_{2})) \\
    &= \dim_{\mathbb{F}} ((V_{1} \cap V_{2})/\mathfrak{m}(V_{1} \cap V_{2})) - \dim_{\mathbb{F}} \Ker \overline{\iota} + \mu_{R}(V_{1} + V_{2}) \\
    &= \mu_{R}(V_{1} \cap V_{2}) + \mu_{R}(V_{1} + V_{2}) - \dim_{\mathbb{F}} \Ker \overline{\iota},
  \end{align*}
  and thus we obtain $\mu_{R}(V_{1}) + \mu_{R}(V_{2}) \leq \mu_{R}(V_{1} + V_{2}) + \mu_{R}(V_{1} \cap V_{2})$.
  Obviously, the equality holds if and only if
  $\dim_{\mathbb{F}} \Ker \overline{\iota} = 0$.
  For $\bm{v} \in V_{1} \cap V_{2}$, $\overline{\iota}(\overline{\bm{v}}) = (\overline{\bm{v}}, -\overline{\bm{v}})$
  is zero if and only if $\bm{v} \in \mathfrak{m}V_{1} \cap \mathfrak{m}V_{2}$.
  Thus,
  \begin{equation*}
    \Ker(\overline{\iota}) = \frac{(V_{1} \cap V_{2}) \cap \mathfrak{m}V_{1} \cap \mathfrak{m}V_{2}}{\mathfrak{m}(V_{1} \cap V_{2})} = \frac{\mathfrak{m}V_{1} \cap \mathfrak{m}V_{2}}{\mathfrak{m}(V_{1} \cap V_{2})}.
  \end{equation*}
  Therefore, the claim $\Ker(\overline{\iota}) = \{ \bm{0} \}$ is
  equivalent to $\mathfrak{m}V_{1} \cap \mathfrak{m}V_{2} = \mathfrak{m}(V_{1} \cap V_{2})$.
\end{proof}

We write $V_{S} \coloneqq \langle A(S) \rangle_{R}$
for all subsets $S \subseteq E$.
Then the rank function of $M \lbrack A \rbrack$ is calculated as follows.

\begin{prop} \label{prop:rank_function}
  For every $X \subseteq E$,
  \begin{equation*}
    r_{M \lbrack A \rbrack}(X) = \max_{I \subseteq X} \mu_{R}(V_{I}).
  \end{equation*}
\end{prop}

\begin{proof}
  By the definition of $r_{M\lbrack A \rbrack}$ and Lemma~\ref{lem:modular_indep_minimal_gen}, we have
  \begin{align*}
    r_{M\lbrack A \rbrack}(X)
    &= \max \{ \abs{I} : I \subseteq X, \, \text{$A_{I}$ is modular independent} \} \\
    &= \max \{ \abs{I} : I \subseteq X, \, \mu_{R}(V_{I}) = \abs{I} \}.
  \end{align*}
  For any $I\subseteq X$ with $\mu_{R}(V_{I}) = \abs{I}$, we have
  $\displaystyle \abs{I} = \mu_{R}(V_{I}) \leq \max_{J\subseteq X}\mu_{R}(V_{J})$.
  Hence
  \begin{equation*}
    r_{M\lbrack A \rbrack}(X)
    \leq \max_{I \subseteq X} \mu_{R}(V_{I}).
  \end{equation*}
  Since $X$ is finite, choose $I_{0}\subseteq X$ such that
  \begin{equation*}
    \mu_{R}(V_{I_{0}})= \max_{I \subseteq X} \mu_{R}(V_{I}) \eqqcolon k.
  \end{equation*}
  We show that there exists $J \subseteq X$ such that $\mu_{R}(V_{J})=\abs{J}=k$.

  Set $\bm{a}_{i} \coloneqq A(i)$ for $i \in I_{0}$, and let $\mathbb{F}=R/\mathfrak{m}$.
  Then $\{ \overline{\bm{a}}_{i} \coloneqq \bm{a}_{i} + \mathfrak{m}V_{I_{0}} \}_{i \in I_{0}}$
  spans the $\mathbb{F}$-vector space $V_{I_{0}}/\mathfrak{m}V_{I_{0}}$.
  Since $\dim_{\mathbb{F}} (V_{I_{0}}/\mathfrak{m}V_{I_{0}}) = \mu_{R}(V_{I_{0}}) = k$,
  there exists $J \subseteq I_{0}$ such that $\{ \overline{\bm{a}}_{j} \}_{j \in J}$
  is a basis for $V_{I_{0}}/\mathfrak{m}V_{I_{0}}$ and $\abs{J} = k$.
  By Theorem~\ref{thm:number_of_minimal_generators}(i),
  $\{ \bm{a}_{j} \}_{j \in J}$ is a minimal set of generators of $V_{I_{0}}$.
  In particular, $V_J=V_{I_0}$, and hence
  \begin{equation*}
    \mu_R(V_J)=\mu_R(V_{I_0})=k=\abs{J}.
  \end{equation*}
  Therefore,
  \begin{equation*}
    r_{M\lbrack A \rbrack}(X) \ge \abs{J}=k=\max_{I \subseteq X} \mu_{R}(V_{I}).
  \end{equation*}
  Together with the opposite inequality proved above,
  the proposition follows.
\end{proof}

Note that $r_{M[A]}(X) \neq \mu_{R}(V_{X})$ in general.
Since $\mu_{R}(V_{\emptyset}) = \mu_{R}(0) = 0$,
the condition (R1) is immediate.
For $e \in E$, put $\bm{a}_{e} \coloneqq A(e)$.
For any $X \subseteq E$,
let $\{ \bm{v}_{1}, \dots, \bm{v}_{\mu_{R}(V_{X})} \}$ be
a minimal generating set of $V_{X}$.
Then $V_{X \cup \{ e \}} = V_{X} + R\bm{a}_{e}$ is generated by
$\{ \bm{v}_{1},\dots, \bm{v}_{\mu_{R}(V_{X})}, \bm{a}_{e} \}$,
hence $\mu_{R}(V_{X \cup \{ e \}}) \leq \mu_{R}(V_{X}) + 1$;
thus (R2+) holds.
Moreover, by the hereditary property of modular independence,
the condition (H) is obvious.
However, as the following example shows, even if $\mu_{R}$ is modular,
not only can submodularity (R3) fail;
even monotonicity (R2) is not assured.

\begin{eg} \label{eg:not_monotonic}
  $R \coloneqq \mathbb{F}\lbrack x, y \rbrack/\langle x^{2}, y^{2}, xy \rangle$,
  where $\mathbb{F}$ is a field, is a local commutative ring,
  but not Frobenius (for example, see \cite[Example~2.4]{Dougherty}).
  Then, $\mathfrak{m} = \langle x, y \rangle = \{ ax + by \mid a, b \in \mathbb{F} \} \cong \mathbb{F}^{2}$.
  $V \coloneqq R$ is the rank $1$ free $R$-module.
  Set $E \coloneqq \{ 1, 2, 3, 4 \}$,
  \begin{equation*}
    A = \bordermatrix{
      & 1 & 2 & 3 & 4 \cr
      & x & y & 1 + x & 1 + y \cr
    } \in V^{E},
  \end{equation*}
  and then fix $X = \{ 1, 2, 3 \}$ and $Y = \{ 1, 2, 4 \}$.
  Since $V_{X} = V_{Y} = R$ by $1 + x, 1 + y \in R^{\times}$,
  we have $V_{X} \cap V_{Y} = R$. Hence,
  \begin{equation*}
    \mathfrak{m}V_{X} \cap \mathfrak{m}V_{Y} = \mathfrak{m}R \cap \mathfrak{m}R = \mathfrak{m}R =   \mathfrak{m}(V_{X} \cap V_{Y}).
  \end{equation*}
  Thus, by Lemma~\ref{lem:muR_supermodular},
  $\mu_{R}$ behaves modularly for this pair.

  On the other hand, $X \cap Y = \{ 1, 2 \}$ implies
  $V_{X \cap Y} = \langle x, y \rangle = \mathfrak{m}$
  and so $V_{X \cap Y} \leq V_{X} \cap V_{Y}$.
  Then, while $X \cap Y \subseteq X, Y$,
  the function $\mu_{R}(V_{-})$ is not monotonic because 
  \begin{equation*}
    \mu_{R}(V_{X \cap Y}) = 2
    > 1 = \mu_{R}(V_{X}) = \mu_{R}(V_{Y}).
  \end{equation*}
  Furthermore, $\mu_{R}(V_{-})$ is not even submodular because
  \begin{equation*}
    \mu_{R}(V_{X}) + \mu_{R}(V_{Y}) = 1 + 1 \leq 1 + 2 = \mu_{R}(V_{X \cup Y}) + \mu_{R}(V_{X \cap Y}).
  \end{equation*}
\end{eg}

However, by Lemma~\ref{lem:modular_indep_minimal_gen}, 
we know that for every $X \subseteq E$,
$A_{X}$ is modular independent if and only if
$r_{M\lbrack A \rbrack}(X) = \abs{X}$.
Hence, as soon as $\mu_{R}(V_{-})$ satisfies monotonicity (R2),
we have $r_{M[A]} = \mu_{R}(V_{-})$; in other words,
the rank function of the independence system coincides with $\mu_{R}(V_{-})$.

The following theorem naturally leads us to the discussion of chain rings.
Note that the theorem does not assume finiteness of $R$.

\begin{thm} \label{thm:monotonic_iff_chain_ring}
  Let $(R, \mathfrak{m})$ be a local commutative ring
  with $\bigcap_{\ell \geq 1} \mathfrak{m}^{\ell} = \{ 0 \}$.
  Then the following are equivalent:
  \begin{enumerate}[label=\textup{(\arabic*)}, itemsep=0ex]
    \item $R$ is a chain ring.
    \item For all finitely generated $R$-modules $V$, $V^{\prime}$
      with $V^{\prime} \leq V$, $\mu_{R}(V^{\prime}) \leq \mu_{R}(V)$.
    \item For all finitely generated ideals $\mathfrak{a}$, $\mathfrak{a}^{\prime}$ of $R$
      with $\mathfrak{a}^{\prime} \subseteq \mathfrak{a}$,
      $\mu_{R}(\mathfrak{a}^{\prime}) \leq \mu_{R}(\mathfrak{a})$.
  \end{enumerate}
\end{thm}
\begin{proof}
  (1) $\Rightarrow$ (2):
  By Lemma~\ref{lem:noetherian_chain_iff_local_PIR}, $R$ is PIR.
  By Lemma~\ref{lem:decomposition_PIR}, every finitely generated $R$-module $V$
  admits an invariant factor decomposition $V \cong \bigoplus_{i = 1}^{k} R/\mathfrak{a}_{i}$,
  where $k = \mu_{R}(V)$. Note that each factor $R/\mathfrak{a}_{i}$
  is cyclic because $\langle 1 + \mathfrak{a}_{i} \rangle = R/\mathfrak{a}_{i}$.
  We prove by induction on $k$ that every finitely generated submodule
  $V^{\prime} \leq V$ satisfies $\mu_{R}(V^{\prime}) \leq k$.
  
  If $k = 1$, then $V$ is cyclic, so every submodule $V^{\prime}$ is cyclic and $\mu_{R}(V^{\prime}) \leq 1$.
  Assume $k \geq 2$ and write $V = U_{k} \oplus R/\mathfrak{a}_{k}$ with $U_{k} = \bigoplus_{i = 1}^{k - 1} R/\mathfrak{a}_{i}$.
  Fix an isomorphism $\varphi \colon V \xrightarrow{\sim} \bigoplus_{i = 1}^{k} R/\mathfrak{a}_{i}$,
  and take the projection
  $\pr_{k} \colon \bigoplus_{i = 1}^{k} R/\mathfrak{a}_{i} \to R/\mathfrak{a}_{k},
  \, (a_{1}, \dots, a_{k}) \mapsto a_{k}$.
  Let $\pi_{k} \colon V \to R/\mathfrak{a}_{k}$ be the composition $\pr_{k} \circ \varphi$ and
  set $W \coloneqq \pi_{k}(V^{\prime}) \leq R/\mathfrak{a}_{k}$.
  Since $R/\mathfrak{a}_{k}$ is cyclic and $R$ is a PIR, $W$ is cyclic;
  choose $\bm{v}_{k} \in V^{\prime}$ with $\pi_{k}(\bm{v}_{k})$ generating $W$.
  Let $V^{\prime\prime} \coloneqq V^{\prime} \cap \Ker \pi_{k} \leq U_{k}$.
  By induction, $\mu_{R}(V^{\prime\prime}) \leq k - 1$.
  For any $\bm{v} \in V^{\prime}$, there exists some $r \in R$
  such that $\pi_{k}(\bm{v}) = r\pi_{k}(\bm{v}_{k})$.
  Since $\pi_{k}(\bm{v} - r\bm{v}_{k}) = 0$ implies $\bm{v} - r\bm{v}_{k} \in V^{\prime} \cap \Ker \pi_{k} = V^{\prime\prime}$,
  we obtain $V^{\prime} = V^{\prime\prime} + \langle \bm{v}_{k} \rangle$.
  Therefore, $\mu_{R}(V^{\prime}) \leq (k - 1) + 1 = k = \mu_{R}(V)$.
  
  (2) $\Rightarrow$ (3) is trivial because ideals are $R$-submodules.

  (3) $\Rightarrow$ (1):
  Assumption (3) applied to a finitely generated ideal $\mathfrak{a} \subseteq R$ gives $\mu_{R}(\mathfrak{a}) \leq \mu_{R}(R) = 1$.
  If $R$ is a field, the claim is obvious, so we may assume $\mathfrak{m} \neq \{ 0 \}$.
  Then, since $\bigcap_{\ell \geq 1} \mathfrak{m}^{\ell} = 0$, we have $\mathfrak{m} \neq \mathfrak{m}^{2}$.
  Take an arbitrary $x \in \mathfrak{m} \setminus \mathfrak{m}^{2}$.
  For any $y \in \mathfrak{m}$, the ideal $\langle x, y \rangle$ is finitely generated,
  and so principal by the assumption.
  Thus, there exists some $d \in R$ with $\langle x, y \rangle = \langle d \rangle$,
  and then $x = rd$ for some $r \in R$.
  If $r \in \mathfrak{m}$, then $x \in \mathfrak{m} \langle d \rangle \subseteq \mathfrak{m}^{2}$; a contradiction.
  Consequently, $r$ is a unit, and so $\langle d \rangle = \langle x \rangle$, which implies $y \in \langle x \rangle$.
  Therefore, the maximal ideal $\mathfrak{m} = \langle x \rangle$ is principal,
  and hence $R$ is a chain ring by Lemma~\ref{lem:noetherian_chain_iff_local_PIR}.
\end{proof}

\begin{cor} \label{cor:rank_when_chain}
  Let $V$ be an $R$-module and $A \colon E \to V$.
  If $R$ is a commutative chain ring with $\bigcap_{\ell \geq 1} \mathfrak{m}^{\ell} = \{ 0 \}$,
  then the rank function of the independence system $M \lbrack A \rbrack$ is
  $r_{M \lbrack A \rbrack} = \mu_{R}(V_{-})$.
\end{cor}
\begin{proof}
  By Theorem~\ref{thm:monotonic_iff_chain_ring}, $\displaystyle \max_{I \subseteq X} \mu_{R}(V_{I}) = \mu_{R}(V_{X})$.
  The corollary follows from Proposition~\ref{prop:rank_function}.
\end{proof}

In the remaining part of this section, we assume that
$R$ is a commutative chain ring with $\bigcap_{\ell \geq 1} \mathfrak{m}^{\ell} = \{ 0 \}$.
Thus it remains to analyse the submodularity of $r_{M\lbrack A \rbrack}$.

\begin{thm} \label{thm:submodular_iff}
  $r_{M\lbrack A \rbrack}$ is submodular if and only if
  \begin{equation*}
    (\mu_{R}(V_{X \cap Y}) =) \dim_{\mathbb{F}} \frac{V_{X \cap Y}}{\mathfrak{m} V_{X \cap Y}} \leq \dim_{\mathbb{F}} \frac{V_{X} \cap V_{Y}}{\mathfrak{m}V_{X} \cap \mathfrak{m}V_{Y}} \quad \text{for all } X, Y \subseteq E.
  \end{equation*}
\end{thm}
\begin{proof}
  As in the proof of Lemma~\ref{lem:muR_supermodular}, we have
  \begin{equation*}
    \mu_{R}(V_{X}) + \mu_{R}(V_{Y}) - \mu_{R}(V_{X} + V_{Y}) = \dim \Im \overline{\iota}.
  \end{equation*}
  Noting that $r_{M\lbrack A \rbrack}(X) = \mu_{R}(V_{X})$ by Corollary~\ref{cor:rank_when_chain}
  and that $V_{X \cup Y} = V_{X} + V_{Y}$, the function $r_{M\lbrack A \rbrack}$ is submodular if and only if
  \begin{equation} \label{eq:dimIm_geq_mu}
    \dim_{\mathbb{F}} \Im \overline{\iota} \geq r_{M[A]}(X \cap Y) = \mu_{R}(V_{X \cap Y}).
  \end{equation}
  Now we calculate $\dim_{\mathbb{F}} \Im \overline{\iota}$. Since
  \begin{equation*}
    \mathfrak{m}(V_{X} \cap V_{Y}) \subseteq \mathfrak{m}V_{X} \cap \mathfrak{m} V_{Y} \subseteq V_{X} \cap V_{Y},
  \end{equation*}
  by the third isomorphism theorem, we have
  \begin{equation*}
    \Ker \left( \frac{V_{X} \cap V_{Y}}{\mathfrak{m}(V_{X} \cap V_{Y})} \to \frac{V_{X} \cap V_{Y}}{\mathfrak{m}V_{X} \cap \mathfrak{m}V_{Y}} \right) = \frac{\mathfrak{m}V_{X} \cap \mathfrak{m}V_{Y}}{\mathfrak{m}(V_{X} \cap V_{Y})}.
  \end{equation*}
  Thus, the short sequence 
  \begin{equation*}
    0 \to \frac{\mathfrak{m}V_{X} \cap \mathfrak{m}V_{Y}}{\mathfrak{m}(V_{X} \cap V_{Y})} \hookrightarrow \frac{V_{X} \cap V_{Y}}{\mathfrak{m}(V_{X} \cap V_{Y})} \twoheadrightarrow \frac{V_{X} \cap V_{Y}}{\mathfrak{m}V_{X} \cap \mathfrak{m}V_{Y}} \to 0
  \end{equation*}
  is exact. Therefore, we obtain
  \begin{align*}
    \dim_{\mathbb{F}} \Im \overline{\iota} &= \mu_{R}(V_{X} \cap V_{Y}) - \dim_{\mathbb{F}} \Ker \overline{\iota} \\
    &= \dim_{\mathbb{F}} \frac{V_{X} \cap V_{Y}}{\mathfrak{m}(V_{X} \cap V_{Y})} - \dim_{\mathbb{F}} \frac{\mathfrak{m}V_{X} \cap \mathfrak{m}V_{Y}}{\mathfrak{m}(V_{X} \cap V_{Y})} = \dim_{\mathbb{F}} \frac{V_{X} \cap V_{Y}}{\mathfrak{m} V_{X} \cap \mathfrak{m}V_{Y}},
  \end{align*}
  as required.
\end{proof}

\begin{cor} \label{cor:sufficient_condition_for_submodularity}
  $r_{M \lbrack A \rbrack}$ is submodular if
  \begin{equation*}
    V_{X \cap Y} \cap \mathfrak{m}V_{X} \cap \mathfrak{m}V_{Y} = \mathfrak{m}V_{X \cap Y} \quad \text{for all} \quad X, Y \subseteq E.
  \end{equation*}
\end{cor}

\begin{proof}
  Fix $X, Y \subseteq E$ . By Theorem~\ref{thm:submodular_iff},
  it is enough to take an injection from
  $V_{X \cap Y}/{\mathfrak{m}V_{X \cap Y}}$ into
  $(V_{X} \cap V_{Y})/(\mathfrak{m}V_{X} \cap \mathfrak{m}V_{Y})$.
  Consider the canonical homomorphism
  \begin{equation*}
    \eta \colon \frac{V_{X \cap Y}}{\mathfrak{m}V_{X \cap Y}} \to \frac{V_{X} \cap V_{Y}}{\mathfrak{m}V_{X} \cap \mathfrak{m}V_{Y}}
  \end{equation*}
  induced by the inclusion $V_{X \cap Y} \hookrightarrow V_{X} \cap V_{Y}$.
  Its kernel is 
  \begin{equation*}
    \Ker \eta \cong \frac{V_{X \cap Y} \cap (\mathfrak{m}{V}_{X} \cap \mathfrak{m}V_{Y})}{\mathfrak{m}V_{X \cap Y}} = \frac{\mathfrak{m}V_{X \cap Y}}{\mathfrak{m}V_{X \cap Y}} = \{ \bm{0} \}
  \end{equation*}
  by the assumption, and so $\eta$ is an injection. Therefore, $r_{M \lbrack A \rbrack}$ is submodular.
\end{proof}
 
\section{Matroid Representation by Codes over Finite Rings} \label{sect:matroid_repr_codes}

In this section we study the independence systems arising from
linear codes $C \leq R^{E}$ via modular independence of the columns
of a generator matrix.
As shown in Section~\ref{sect:matroid_repr_by_modular_indep},
the matroidal behaviour of this construction is most transparent
over chain rings: in particular, chain rings are exactly
the local rings for which $\mu_{R}$ is monotone on submodules,
and in that setting the resulting rank function is governed by $\mu_{R}$.
Moreover, following coding-theoretic convention,
we assume throughout that $R$ is finite, so that codes and their duals
admit generator and parity-check matrices.
Accordingly, for the remainder of this section we work over
a finite commutative chain ring $(R,\mathfrak m)$;
we keep the notation $\mathfrak{m} = \langle\theta\rangle$
for the maximal ideal, write $R/\mathfrak{m} = \mathbb{F}_{q}$
for the residue field, and denote by $\nu$ the nilpotency index
of $\mathfrak{m}$.
For a finite set $J$, setting $V \coloneqq R^{J}$, we may regard
$A \in V^{E} \cong R^{J \times E}$ as a $\abs{J} \times \abs{E}$ matrix over $R$.
Accordingly, we need to establish that the independence system
associated with a code $C$ is independent of the choice of generator matrices for $C$.
For this, it is enough to extend \cite[Lemma~3.4]{DoughertyLiu09}
to a general local ring $R$.
Note that the following Lemma~\ref{lem:modular_independence_is_invariant_under_isomorphism}
and Proposition~\ref{prop:modular_independence_is_invariant_under_elementary_row_operations}
hold over any commutative local ring (no finiteness or chain condition is needed).

\begin{lem} \label{lem:modular_independence_is_invariant_under_isomorphism}
  Let $V$ be an $R$-module and $\varphi \colon V \to V$ be an $R$-isomorphism.
  Then, $\bm{v}_{1}, \dots, \bm{v}_{\ell} \in V$ are modular independent
  if and only if $\varphi(\bm{v}_{1}), \dots, \varphi(\bm{v}_{\ell}) \in V$
  are modular independent.
\end{lem}
\begin{proof}
  For $\bm{v}_{1}, \dots, \bm{v}_{\ell} \in V$, consider a linear relation
  $\displaystyle \sum_{i = 1}^{\ell} \alpha_{i} \varphi(\bm{v}_{i}) = \bm{0}$.
  Applying the inverse $\varphi^{-1} \colon V \to V$, we obtain the linear relation
  $\displaystyle \sum_{i = 1}^{\ell} \alpha_{i} \bm{v}_{i} = \bm{0}$.
  Therefore, if $\bm{v}_{1}, \dots, \bm{v}_{\ell}$ are modular independent,
  then $\alpha_{i} \in \mathfrak{m}$; thus $\varphi(\bm{v}_{1}), \dots, \varphi(\bm{v}_{\ell})$
  are modular independent. Conversely, if $\bm{v}_{1}, \dots, \bm{v}_{\ell}$
  are modular dependent, then $\alpha_{i} \in R^{\times}$ for some $i \in \{ 1, \dots, \ell \}$;
  and hence $\varphi(\bm{v}_{1}), \dots, \varphi(\bm{v}_{\ell})$ are modular dependent.
\end{proof}

\begin{prop} \label{prop:modular_independence_is_invariant_under_elementary_row_operations}
  Let $J$ and $E$ be finite sets, and $G \in R^{J \times E}$ be
  an $\abs{J} \times \abs{E}$ matrix. Then performing operations of the type
  \begin{enumerate}[label=\textup{(R\arabic*)}, itemsep=0ex]
    \item Permutation of the rows,
    \item Multiplication of a row by a unit of $R$,
    \item Addition of a scalar multiple of one row to another
  \end{enumerate}
  preserve $M \lbrack G \rbrack$.
  In particular, if $G$ and $G^{\prime}$ are generator matrices for $C$,
  then $M \lbrack G \rbrack = M \lbrack G^{\prime} \rbrack$.
\end{prop}
\begin{proof}
  The former claim is immediate from Lemma~\ref{lem:modular_independence_is_invariant_under_isomorphism}.
  If $G$ and $G^{\prime}$ are generator matrices for $C$, then
  by Theorem~\ref{thm:number_of_minimal_generators}
  there exists an invertible matrix $P^{\mu_{R}(C) \times \mu_{R}(C)}$
  such that $G^{\prime} = PG$; hence Lemma~\ref{lem:modular_independence_is_invariant_under_isomorphism}
  yields $M\lbrack G \rbrack = M\lbrack G^{\prime} \rbrack$.
\end{proof}

Accordingly, we define $M(C) \coloneqq M \lbrack G \rbrack$,
where $G$ is any generator matrix for $C$.

\subsection{Minors: Deletion and Contraction}

We now study the behaviour of representability under minors.
We first describe the circuits in coding-theoretic terms.

\begin{prop} \label{prop:circuits}
  Let $M(C)$ be the independence system associated with an $R$-code $C \leq R^{E}$.
  Then the collection of circuits of $M(C)$ is
  \begin{equation*}
    \Min(\{ \supp(\bm{x}) \subseteq E \mid \bm{x} \in \psi(C^{\perp}) \}).
  \end{equation*}
\end{prop}
\begin{proof}
  Let $\mathcal{C}$ be the collection of circuits in $M(C)$, and set
  \begin{equation*}
    \mathcal{C}^{\prime} \coloneqq \Min(\{ \supp(\bm{x}) \subseteq E \mid \bm{x} \in \psi(C^{\perp}) \}).
  \end{equation*}
  Let $G \in R^{J \times E}$ be a generator matrix for $C$,
  where $\abs{J} = \mu_{R}(C)$.
  Take an arbitrary circuit $D \in \mathcal{C}$.
  Since $D \notin \mathcal{I}$, the columns of $G_{J, D}$
  are modular dependent. Hence, there exists $(\alpha_{i})_{i \in D} \in R^{\abs{D}}$
  such that $\sum_{i \in D} \alpha_{i} \bm{g}_{i} = \bm{0}$ and
  $\alpha_{u} \in R^{\times}$ for some $u \in D$,
  where $\bm{g}_{i} \coloneqq G_{J, \{ i \}}$ denotes
  the column of $G$ indexed by $i$.
  By minimality of $D$, note that $\alpha_{i} \neq 0$ for all $i \in D$.
  We define a vector $\bm{x} = (x_{i})_{i \in E} \in R^{E}$ as
  \begin{equation*}
    x_{i} \coloneqq \begin{cases}
      \alpha_{i} & \text{if } i \in D \\
      0 & \text{if } i \in E \setminus D
    \end{cases} \quad \text{for all } i \in E.
  \end{equation*}
  Then, clearly $\supp(\bm{x}) = D$ and since
  \begin{equation*}
    G\bm{x}^{\top}
    = \sum_{i \in E} x_{i}\bm{g}_{i}
    = \sum_{i \in D} \alpha_{i} \bm{g}_{i} + \sum_{i \in E \setminus D} 0 \bm{g}_{i}
    = \bm{0},
  \end{equation*}
  we have $\bm{x} \in C^{\perp}$.
  Furthermore, since $x_{u} = \alpha_{u}$ is a unit,
  we have $\bm{x} \in \psi(C^{\perp})$.
  If there exists $\bm{w} \in \psi(C^{\perp})$ such that $\supp(\bm{w}) \subsetneq \supp(\bm{x})$,
  then $G \bm{w}^{\top} = \bm{0}$, and so $D \supsetneq \supp(\bm{w}) \notin \mathcal{I}$
  by Lemma~\ref{lem:psi_and_Ann} and Remark~\ref{rem:psi_and_Ann};
  a contradiction to the minimality of $D$.
  Consequently, we have $\supp(\bm{x}) \in \mathcal{C}^{\prime}$,
  and thus $\mathcal{C} \subseteq \mathcal{C}^{\prime}$.
  Conversely, taking $D \in \mathcal{C}^{\prime}$ and
  $\bm{x} \in \psi(C^{\perp})$ such that $\supp(\bm{x}) = D$,
  there exists a coordinate $i \in E$ where $x_{i} \in R^{\times}$
  by Lemma~\ref{lem:psi_and_Ann}.
  So we may reverse the argument above,
  which gives $\mathcal{C}^{\prime} \subseteq \mathcal{C}$,
  and hence $\mathcal{C} = \mathcal{C}^{\prime}$.
\end{proof}

Actually, Proposition~\ref{prop:circuits} holds for
any local ring satisfying Lemma~\ref{lem:psi_and_Ann}(ii)
(or equivalently Lemma~\ref{lem:psi_and_Ann}(i)).

Next, we establish that puncturing represents deletion of $M(C)$,
whether or not $M(C)$ is a matroid.

\begin{lem} \label{lem:deletion_is_puncturing}
  Let $C \leq R^{E}$ be an $R$-code.
  For all $X \subseteq E$, we have $M(C^{X}) = M(C) \backslash X$.
\end{lem}
\begin{proof}
  Let $G\in R^{J\times E}$ be a generator matrix for $C$,
  so that $C = \rowspan_{R}(G)$.
  Let $\pr_{E \setminus X} \colon R^{E} \to R^{E\setminus X}$
  be the coordinate projection, and let $G^{\prime} \coloneqq G_{J,\,E\setminus X}$
  be the matrix obtained from $G$ by deleting the columns indexed by $X$.
  Then
  \begin{equation*}
    \rowspan_{R}(G^{\prime})=\pr_{E \setminus X}(\rowspan_{R}(G))
    = \pr_{E \setminus X}(C) = C^{X},
  \end{equation*}
  so $G^{\prime}$ is a generating matrix for the punctured code $C^{X}$.

  Although $G^{\prime}$ need not be minimal, we may apply
  elementary row operations to $G^{\prime}$ so that
  \begin{equation*}
    PG^{\prime} = \begin{pmatrix} G^{X} \\ \bm{0}^{\top} \end{pmatrix}
  \end{equation*}
  where the nonzero rows of $G^{X}$ form a minimal generating set of $C^{X}$.
  Since $P$ is invertible, Lemma~\ref{lem:modular_independence_is_invariant_under_isomorphism}
  implies $M\lbrack PG^{\prime} \rbrack = M\lbrack G^{\prime} \rbrack$.
  Moreover, adjoining or deleting zero rows does not affect
  modular independence of columns,
  hence $M\lbrack PG^{\prime} \rbrack = M\lbrack G^{X} \rbrack =M(C^{X})$.
  Therefore $M(C^{X}) = M\lbrack G^{\prime} \rbrack$.

  Finally, for any $I\subseteq E\setminus X$,
  the submatrix $G_{J,I}^{\prime}$ coincides with $G_{J,I}$.
  Thus the columns indexed by $I$ are modular independent in $G^{\prime}$
  if and only if they are modular independent in $G$.
  Equivalently, the independent sets of $M\lbrack G^{\prime} \rbrack$
  are exactly the independent sets of $M\lbrack G \rbrack$
  contained in $E\setminus X$, i.e.
  $M\lbrack G^{\prime} \rbrack = M\lbrack G \rbrack \setminus X$.
  Since $M(C) = M\lbrack G \rbrack$, we conclude $M(C^{X}) = M(C) \setminus X$.
\end{proof}

The proof of Lemma~\ref{lem:deletion_is_puncturing} uses
only the definition of $M(C)$ in terms of modular independence
of the columns of a generator matrix, and therefore
does not require any finiteness or chain-ring assumption on $R$
(beyond locality and finite generation of $C$).
In general, shortening does not realise contraction,
even over a finite chain ring and even when
$M(C)$ is a matroid.

\begin{eg}\label{eg:shortening_not_representing_contraction}
  Let $R=\mathbb{Z}_4$ and let $C\le R^{\{1,2\}}$ be the code generated by
  \begin{equation*}
    G = \bordermatrix{
      & 1 & 2 \cr
      & 2 & 1 \cr
      & 0 & 2
    }.
  \end{equation*}
  Each singleton is modular independent, whereas
  \begin{equation*}
    1 \cdot \begin{pmatrix} 2 \\ 0 \end{pmatrix}
    + 2 \begin{pmatrix} 1 \\ 2 \end{pmatrix}
    = \begin{pmatrix} 0 \\ 0 \end{pmatrix},
  \end{equation*}
  so $\{ 1, 2 \}$ is the unique circuit of $M(C)$.
  Hence $M(C) \cong U_{1,2}$. Now
  \begin{equation*}
    C_{\{1\}}
    =\{(c_{2})\in \mathbb{Z}_{4} \mid (0, c_{2}) \in C \}
    =\langle (2) \rangle,
  \end{equation*}
  so $M(C_{\{1\}})$ has no circuits on the ground set $\{ 2 \}$.
  On the other hand, $M(C)/\{1\}$ is isomorphic to the uniform matroid $U_{0, 1}$,
  where the unique circuit is $\{ 2 \}$. Therefore,
  \begin{equation*}
    M(C_{\{1\}})\neq M(C)/\{1\}.
  \end{equation*}
\end{eg}

\begin{eg} \label{eg:shortening_of_matroid_is_not_matroid}
  Let $R = \mathbb{Z}_{4}$ and $C \leq R^{\lbrack 4 \rbrack}$ be
  the code generated by
  \begin{equation*}
    G = \bordermatrix{
      & 1 & 2 & 3 & 4 \cr
      & 2 & 1 & 0 & 2 \cr
      & 0 & 0 & 1 & 2
    }.
  \end{equation*}
  $M(C)$ is a matroid which is a $3$-point line with one double point.
  Then shortened code $C_{\{ 4 \}}$ has a generator matrix
  \begin{equation*}
    G^{\prime} = \bordermatrix{
      & 1 & 2 & 3 \cr
      & 2 & 1 & 1 \cr
      & 0 & 0 & 2
    }.
  \end{equation*}
  This shortening is not compatible with contraction
  because $M(C_{\{ 4 \}})$ is not even a matroid
  as in Example~\ref{eg:non_matroid}.
\end{eg}

To describe shortening in terms of a generator matrix,
we introduce the following operation.

\begin{dfn} \label{dfn:modular-independent_row_reduction}
  Let $G \in (R^{k})^{E}$ be a matrix over $R$, and $e \in E$.
  A \emph{minimal-generator reduction on $e$} is a sequence of
  row operations so that non-zero entries in the column indexed by $e$
  are modular independent.
\end{dfn}

For $e \in E$, let
\begin{equation*}
  \pr_{e} \colon C \to R, \; \bm{c} = (c_{i})_{i \in E} \mapsto c_{e}
\end{equation*}
denote the $e$-th coordinate map.
The entries in the $e$-th column generate $\pr_{e}(C)$,
and by Theorem~\ref{thm:number_of_minimal_generators}
one may apply invertible row operations
so that the non-zero entries in that column form
a minimal generating set of $\pr_{e}(C)$.

\begin{rem} \label{rem:modular-independent_row_reduction}
  Over a chain ring, as in the field case,
  the resulting matrix of a minimal-generator reduction
  has only one non-zero entry
  in the column indexed by $e$,
  as in Lemma~\ref{lem:generator_matrix_for_shortening} below.
  For example, consider the following matrix over $\mathbb{Z}_{4}$:
  \begin{equation*}
    \begin{pmatrix}
      1 & 1 & 2 \\
      0 & 2 & 2
    \end{pmatrix}.
  \end{equation*}
  Performing a minimal-generator reduction on the second column,
  we obtain the following matrix:
  \begin{equation*}
    \begin{pmatrix}
      1 & 1 & 2 \\
      2 & 0 & 2
    \end{pmatrix}.
  \end{equation*}
  Applying a minimal-generator reduction to the third column yields, for example,
  \begin{equation*}
    \begin{pmatrix}
      1 & 3 & 0 \\
      0 & 2 & 2  
    \end{pmatrix}
    \quad \text{or} \quad
    \begin{pmatrix}
      1 & 1 & 2 \\
      1 & 3 & 0
    \end{pmatrix},
  \end{equation*}
  showing that a minimal-generator reduction need not be unique.
\end{rem}

\begin{lem}\label{lem:generator_matrix_for_shortening}
  Let $R$ be a finite commutative chain ring with maximal ideal
  $\mathfrak{m} = \langle \theta \rangle$ and nilpotency index $\nu$.
  Let $C \leq R^{E}$ be an $R$-code with generator matrix $G$,
  and let $e\in E$.
  Assume that the column of $G$ indexed by $e$ is non-zero.
  If $G_{0}$ is obtained from $G$ by a minimal-generator reduction
  on the column indexed by $e$ and a permutation of the rows, then
  \begin{equation*}
    G_{0} =
    \begin{pmatrix}
      u\theta^{t} & \bm{a}^{\top} \\
      0 & G_{0}^{\prime}
    \end{pmatrix}
  \end{equation*}
  for some $u \in R^{\times}$ and $0 \leq t \leq \nu - 1$.
  Define
  \begin{equation*}
    \bm{x}_{e} \coloneqq
    \theta^{\nu - t}u^{-1} \bm{a} \in R^{E \setminus \{ e \}},
    \qquad
    G_{\{ e \}} \coloneqq
    \begin{pmatrix}
      \bm{x}_{e}^{\top} \\
      G_{0}^{\prime}
    \end{pmatrix}.
  \end{equation*}
  Then the row span of $G_{\{e\}}$ is the shortened code $C_{\{e\}}$.
\end{lem}

\begin{proof}
  Since $R$ is a chain ring, any two non-zero elements of $R$ are modular dependent.
  Hence, after a minimal-generator reduction on the column indexed by $e$,
  that column has exactly one non-zero entry.
  After permuting the rows, we may therefore write $G_{0}$ in the displayed block form.

  Every codeword in the row span of $G_{0}$ is of the form
  \begin{equation*}
    (\lambda u\theta^t,\ \lambda \bm a+\bm y),
    \quad \text{where }
    \lambda\in R,\ \bm{y} \in \rowspan_{R}(G_{0}^{\prime}).
  \end{equation*}
  Its coordinate indexed by $e$ is zero if and only if
  $\lambda \in \Ann_{R}(u\theta^{t}) = \langle \theta^{\nu-t} \rangle$.
  Thus $\lambda = \alpha \theta^{\nu-t}u^{-1}$ for some $\alpha \in R$, and then
  \begin{equation*}
    \lambda\bm{a} + \bm{y} = \alpha \bm{x}_{e} + \bm{y}.
  \end{equation*}
  Therefore the puncturings of the codewords
  whose coordinate indexed by $e$ is zero are exactly
  the row span of $G_{\{ e \}}$.
\end{proof}

Now we introduce the notion of contractibility.

\begin{dfn}\label{dfn:contractibility}
  Let $C \leq R^{E}$ be an $R$-code. 
  Assume that $\pr_{e}(C) \neq \{ 0 \}$. Since $R$ is a chain ring,
  $\pr_{e}(C)$ is principal, so we may write
  \begin{equation*}
    \pr_{e}(C) = \langle \theta^{t} \rangle,
    \quad \text{where} \quad
    0 \leq t < \nu.
  \end{equation*}
  We say that $C$ is \emph{contractible by $e$}
  if there exists a codeword
  $\bm{c} = (c_{i})_{i \in E} \in C$ such that
  \begin{equation*}
    c_{e} = u \theta^{t}
    \quad\text{for some } u\in R^{\times},
    \qquad\text{and}\qquad
    c_{i} \in \langle \theta^{t}\rangle
    \quad\text{for all } i \in E.
  \end{equation*}
\end{dfn}

\begin{lem} \label{lem:shortening_with_matrix}
  Use the same notation as in Lemma~\ref{lem:generator_matrix_for_shortening},
  and write $\pr_{e}(C)=\langle \theta^{t}\rangle$.
  Then the following are equivalent:
  \begin{enumerate}[label=\textup{(\roman*)}, itemsep=0ex]
    \item $C$ is contractible by $e$.
    \item $C$ has a generator matrix of the form
      \begin{equation*}
        G_{0} =
        \begin{pmatrix}
        u\theta^{t} & \bm{a}^{\top} \\
        0 & G_{0}^{\prime}
        \end{pmatrix}
      \end{equation*}
      with $u \in R^{\times}$ and every entry of $a$
      belonging to $\langle \theta^{t} \rangle$.
  \end{enumerate}
  When these conditions hold, we write $G/\{ e \} \coloneqq G_{0}^{\prime}$
  for such a choice of $G_{0}$.
\end{lem}
\begin{proof}
  (ii) $\Rightarrow$ (i) is immediate,
  since the first row of $G_{0}$ is a codeword
  satisfying the condition in Definition~\ref{dfn:contractibility}.
  To prove (i) $\Rightarrow$ (ii), let $\bm{c} = (c_{i})_{i \in E}\in C$
  be a codeword as in Definition~\ref{dfn:contractibility}, so that
  \begin{equation*}
    c_{e} = u\theta^t \quad\text{with } u\in R^{\times},
    \qquad
    c_{i} \in \langle \theta^{t}\rangle
    \quad\text{for all } i \in E.
  \end{equation*}
  Now $c_{e} = u\theta^{t} \notin \mathfrak{m}\pr_{e}(C) = \langle \theta^{t + 1} \rangle$,
  implies $\bm{c} \notin \mathfrak{m}C$. 
  Since $\bm{c} + \mathfrak{m}C \in (C/\mathfrak{m}C) \setminus \{ \bm{0} \}$,
  we may extend $\bm{c}$ to a minimal generating set for $C$ by Theorem~\ref{thm:number_of_minimal_generators}(i).
  Thus, we may choose a generator matrix of $C$ whose first row is $\bm{c}$.
  Since
  \begin{equation*}
    \pr_{e}(C) = \langle c_{e} \rangle = \langle \theta^{t} \rangle,
  \end{equation*}
  the $e$-th coordinate of every other row lies in $\langle c_{e} \rangle$.
  Hence, by subtracting suitable multiples of the first row,
  we may make the $e$-th coordinate of every other row equal to $0$.
  This yields a generator matrix of the required form.
\end{proof}

When $G$ is contractible by $e$, the first row of $G_{\{ e \}}$ is zero,
so deleting that row yields $G/\{ e \} = G_{0}^{\prime}$.
We also remark that $G$ is contractible by any element if $R$ is a field.
Then we have the following theorem as justification for our terminology:

\begin{thm}\label{thm:contraction_is_shortening_if_contractible}
  Let $R$ be a finite commutative chain ring,
  $C \leq R^{E}$ be an $R$-code, and $e \in E$.
  If $C$ is contractible by $e$, then
  \begin{equation*}
    M(C_{\{ e \}}) = M\lbrack G/\{ e \}\rbrack = M(C)/\{ e \}.
  \end{equation*}
\end{thm}

\begin{proof}
  By Lemma~\ref{lem:shortening_with_matrix}, we may choose
  a generator matrix for $C$ of the form
  \begin{equation*}
    G_{0} =
      \begin{pmatrix}
        u\theta^{t} & \bm{a}^{\top} \\
        0 & G_{0}^{\prime}
      \end{pmatrix},
  \end{equation*}
  where $u \in R^{\times}$, $\pr_{e}(C) = \langle \theta^{t} \rangle$,
  and every entry of $\bm{a}$ belongs to $\langle \theta^{t} \rangle$.

  We first show that $C_{\{ e \}}$ is generated by the rows of $G_{0}^{\prime}$.
  Let $\bm{x} = (u\theta^{t}, \bm{a}^{\top})$ be the first row of $G_{0}$,
  and let $K$ be the row span of $(0 \; G_{0}^{\prime})$.
  Then $C = \langle \bm{x} \rangle + K$.
  Take $\bm{c} \in C$ with $c_{e} = 0$. We may write
  $\bm{c} = \lambda \bm{x} + \bm{y}$ with $\lambda \in R,\ \bm{y} \in K$.
  Since every element of $K$ has $e$-th coordinate $0$,
  the condition $c_{e} = 0$ gives $\lambda u\theta^{t} = 0$.
  As $u$ is a unit, this is equivalent to
  \begin{equation*}
    \lambda\in \Ann(\theta^{t})=\langle \theta^{\nu-t}\rangle.
  \end{equation*}
  Because every entry of $\bm{a} = (a_{i})_{i \in E \setminus \{ e \}}$
  lies in $\langle \theta^{t} \rangle$,
  we have $\lambda \bm{a} = \bm{0}$.
  Therefore, after deleting the $e$-th coordinate, the vector $\bm{c}$
  is represented by an element of the row span of $G_{0}^{\prime}$.
  Hence
  \begin{equation*}
    C_{\{ e \}} = \rowspan_{R}(G_{0}^{\prime}),
  \end{equation*}
  and so
  \begin{equation*}
    M(C_{\{e\}}) = M\lbrack G_{0}^{\prime} \rbrack
    = M\lbrack G/\{ e \} \rbrack.
  \end{equation*}

  It remains to prove that
  \begin{equation*}
    M\lbrack G_{0}^{\prime} \rbrack
    = M\lbrack G_{0} \rbrack/\{ e \}.
  \end{equation*}
  Let $\mathcal{C}$ and $\mathcal{C}^{\prime}$ denote
  the collection of circuits of $M\lbrack G_{0} \rbrack$ and
  $M\lbrack G_{0}^{\prime} \rbrack$, respectively.
  By Definition~\ref{dfn:deletion_and_contraction},
  it suffices to show that
  $\mathcal{C}^{\prime} = \mathcal{C}/\{ e \}$.

  Write
  \begin{equation*}
    \bm{g}_{i} = \begin{cases}
      (u\theta^{t}, 0, \dots, 0)^\top & \text{if } i = e, \\
      (a_{i}, \bm{g}_{i}^{\prime \top})^{\top} & \text{if } i \in E \setminus \{ e \},
    \end{cases}
  \end{equation*}
  where $\bm{g}_{i}^{\prime}$ is the column of $G_{0}^{\prime}$ indexed by $i$,
  and $a_{i} \in \langle \theta^{t} \rangle$ for all $i \in E \setminus \{ e \}$.

  We first prove the following two claims.

  \begin{claim} \label{claim:D0_minus_e_is_included_in_S}
    If $S \subseteq E \setminus\{ e \}$ is dependent in $M\lbrack G_{0}^{\prime} \rbrack$,
    then there exists $D \in\mathcal{C}$ such that
    \begin{equation*}
      D \setminus \{ e \}\subseteq S.
    \end{equation*}
  \end{claim}
  \begin{proof}[Proof of Claim~\ref{claim:D0_minus_e_is_included_in_S}]
    Choose coefficients $(\alpha_i)_{i\in S}$, with some
    $\alpha_{j} \in R^{\times}$, such that
    $\displaystyle \sum_{i\in S}\alpha_{i} \bm{g}_{i}^{\prime} = \bm{0}$.
    Since $a_{i} \in \langle \theta^{t}\rangle$, we may write
    $a_{i} = \theta^{t} a_{i}^{\prime}$ ($i\in S$)
    for some $a_{i}^{\prime} \in R$. Set
    \begin{equation*}
      \beta \coloneqq -u^{-1}\sum_{i\in S} \alpha_{i} a_{i}^{\prime}.
    \end{equation*}
    Then
    \begin{equation*}
      \sum_{i\in S}\alpha_{i} \bm{g}_{i} + \beta \bm{g}_{e} = 0.
    \end{equation*}
    Thus $S \cup \{ e \}$ is dependent in $M\lbrack G_{0} \rbrack$,
    so it contains a circuit $D \in \mathcal{C}$.
    Therefore $D \setminus \{ e \} \subseteq S$
    which proves Claim~\ref{claim:D0_minus_e_is_included_in_S}.
  \end{proof}

  \begin{claim} \label{claim:D0_minus_e_is_dependent_in_MB}
    If $D \in \mathcal{C}$, then $D \setminus \{ e \}$ is
    dependent in $M\lbrack G_{0}^{\prime} \rbrack$.
  \end{claim}
  \begin{proof}[Proof of Claim~\ref{claim:D0_minus_e_is_dependent_in_MB}]
    If $e \notin D$, this is immediate from the lower block
    of a modular dependence relation on $D$.
    Suppose $e \in D$. Choose coefficients $(\gamma_{i})_{i\in D \setminus \{ e \}}$
    and $\delta$, with at least one of them a unit, such that
    \begin{equation*}
      \sum_{i \in D \setminus \{ e \}} \gamma_{i} \bm{g}_{i} + \delta \bm{g}_{e} = \bm{0}.
    \end{equation*}
    If $\gamma_{i} \in \mathfrak{m}$ for all $i \in D \setminus \{ e \}$,
    then
    \begin{equation*}
      \sum_{i \in D \setminus \{ e \}} \gamma_{i} a_{i} \in
      \mathfrak{m} \langle \theta^{t} \rangle
      =
      \langle \theta^{t+1}\rangle.
    \end{equation*}
    Hence the first coordinate gives
    \begin{equation*}
      u\theta^{t}\delta
      =
      -\sum_{i\in D \setminus \{ e \}} \gamma_{i} a_{i}
      \in \langle \theta^{t+1}\rangle.
    \end{equation*}
    Since $u \theta^{t}$ generates $\langle \theta^{t} \rangle$,
    this forces $\delta \in \mathfrak{m}$, contradicting
    the choice of the coefficients. Therefore
    $\gamma_{i} \in R^{\times}$ for some $i \in D \setminus\{ e \}$.
    Looking at the lower block, we obtain
    \begin{equation*}
      \sum_{i \in D \setminus \{ e \}}
        \gamma_{i} \bm{g}_{i}^{\prime} = \bm{0},
    \end{equation*}
    so $D \setminus \{ e \}$ is dependent in $M\lbrack G_{0}^{\prime} \rbrack$,
    which proves Claim~\ref{claim:D0_minus_e_is_dependent_in_MB}.
  \end{proof}

  Now let $D^{\prime} \in \mathcal{C}^{\prime}$.
  By Claim~\ref{claim:D0_minus_e_is_included_in_S},
  there exists $D \in \mathcal{C}$ such that
  $D \setminus \{ e \} \subseteq D^{\prime}$.
  By Claim~\ref{claim:D0_minus_e_is_dependent_in_MB},
  the set $D \setminus \{ e \}$ is dependent in $M\lbrack G_{0}^{\prime} \rbrack$.
  Since $D^{\prime}$ is a circuit of $M\lbrack G_{0}^{\prime} \rbrack$,
  we must have $D \setminus \{ e \} = D^{\prime}$.
  Moreover, if $D_{0}^{\prime} \subsetneq D^{\prime}$ were of the form
  $D_{0}^{\prime} = D_{0} \setminus \{ e \}$ for some
  $D_{0} \in \mathcal{C}$, then Claim~\ref{claim:D0_minus_e_is_dependent_in_MB}
  would imply that $D_{0}^{\prime}$ is dependent in $M\lbrack G_{0}^{\prime} \rbrack$,
  a contradiction. Hence $D^{\prime} \in \mathcal{C}/\{ e \}$.

  Conversely, let $D^{\prime} \in \mathcal{C}/\{ e \}$.
  Then $D^{\prime} = D \setminus \{ e \}$ for some $D \in \mathcal{C}$,
  and $D^{\prime}$ is minimal among such sets.
  By Claim~\ref{claim:D0_minus_e_is_dependent_in_MB},
  $D^{\prime}$ is dependent in $M\lbrack G_{0}^{\prime} \rbrack$.
  If some proper subset $D_{0}^{\prime} \subsetneq D^{\prime}$
  were dependent in $M\lbrack G_{0}^{\prime} \rbrack$,
  then Claim~\ref{claim:D0_minus_e_is_included_in_S} would yield
  a circuit $D_{0} \in \mathcal{C}$ such that
  \begin{equation*}
    D_{0} \setminus \{ e \} \subseteq D_{0}^{\prime} \subsetneq D^{\prime},
  \end{equation*}
  contradicting the minimality of $D^{\prime}$.
  Therefore every proper subset of $D^{\prime}$ is independent in
  $M\lbrack G_{0}^{\prime} \rbrack$, so $D^{\prime} \in \mathcal{C}^{\prime}$.

  Thus $\mathcal{C}^{\prime} = \mathcal{C}/\{e\}$,
  whence $M\lbrack G_{0}^{\prime} \rbrack = M\lbrack G_{0} \rbrack/\{ e \}$.
  Finally, Proposition~\ref{prop:modular_independence_is_invariant_under_elementary_row_operations}
  gives $M\lbrack G_{0} \rbrack = M(C)$.
  Therefore,
  \begin{equation*}
    M(C_{\{ e \}}) = M\lbrack G/\{ e \} \rbrack = M(C)/\{ e \}. \qedhere
  \end{equation*}
\end{proof}

We now extend contractibility from single coordinates to subsets.

\begin{dfn}\label{dfn:contractible_subset}
  Let $C \leq R^{E}$ be an $R$-code, and $X \subseteq E$.
  We say that $C$ is \emph{contractible} by $X$
  if there exists an ordering $X = \{ x_{1}, \dots, x_{\ell} \}$ such that,
  defining recursively
  \begin{equation*}
    C^{(0)} \coloneqq C, \qquad
    C^{(j)} \coloneqq \bigl(C^{(j-1)}\bigr)_{\{ x_{j} \}}
    \quad (j \in \lbrack \ell \rbrack),
  \end{equation*}
  the code $C^{(j-1)} \leq R^{E \setminus \{ x_{1}, \dots, x_{j-1} \}}$
  is contractible by $x_{j}$ for every $j \in \lbrack \ell \rbrack$.
\end{dfn}

\begin{cor}\label{cor:contractible_subset}
  Let $X \subseteq E$. If $C$ is contractible by $X$,
  then $M(C_{X}) = M(C)/X$.
\end{cor}
\begin{proof}
  Choose an ordering $X = \{ x_{1}, \dots, x_{\ell} \}$ as in
  Definition~\ref{dfn:contractible_subset}, and define codes
  $C^{(j)}$ recursively by
  \begin{equation*}
    C^{(0)} \coloneqq C, \qquad
    C^{(j)} \coloneqq \bigl(C^{(j-1)}\bigr)_{\{ x_{j} \}}
    \quad (j \in \lbrack \ell \rbrack).
  \end{equation*}
  For each $j$, the code $C^{(j-1)}$ is contractible by $x_{j}$.
  Hence Theorem~\ref{thm:contraction_is_shortening_if_contractible} gives
  \begin{equation*}
    M\bigl(C^{(j)}\bigr) = M\bigl(C^{(j-1)}\bigr)/\{ x_{j} \}
    \qquad (j \in \lbrack \ell \rbrack).
  \end{equation*}
  Iterating, we obtain
  \begin{equation*}
    M\bigl(C^{(\ell)}\bigr)
    = M(C)/\{ x_{1} \}/\cdots/\{ x_{\ell}\}
    = M(C)/X.
  \end{equation*}
  On the other hand, by the definition of shortening,
  a codeword survives the successive shortenings by $x_{1}, \dots, x_{\ell}$
  if and only if all coordinates in $X$ are zero. Therefore
  \begin{equation*}
    C^{(\ell)}
    = \bigl(\cdots((C_{\{ x_{1} \}})_{\{ x_{2} \}})\cdots\bigr)_{\{ x_{\ell} \}}
    = C_{X}.
  \end{equation*}
  Hence $M(C_{X}) = M(C)/X$.
\end{proof}

\subsection{Duality}

Next we turn to duality.
Over a finite field, if a linear code $C$ represents a matroid,
then the dual code $C^{\perp}$ represents the dual matroid;
equivalently,
\begin{equation*}
  M(C^{\perp})=M(C)^{\ast}.
\end{equation*}
We now investigate to what extent this relation survives
for independence systems arising from modular independence
over finite chain rings.
Since a general independence system does not come equipped with
a satisfactory notion of duality defined via bases or circuits
alone, we begin by defining a dual independence system directly
from the rank function.
This construction is motivated by the usual dual rank formula
for matroids and by the duality theory of demi-matroids (see \cite{DemiMatroid}).

\begin{dfn} \label{dfn:dual_of_independence_system}
  Let $M = (E, \mathcal{I})$ be an independence system with rank function $r_{M}$.
  Define an auxiliary function $r_{M}^{\perp} \colon 2^{E} \to \mathbb{Z}$ as
  \begin{equation*}
    r_{M}^{\perp}(X) \coloneqq \abs{X} + r_{M}(E \setminus X) - r_{M}(E) \quad
    \text{for all } X \subseteq E.
  \end{equation*}
  Set
  \begin{equation*}
    \mathcal{I}^{\ast} \coloneqq \{ X \subseteq E : r_{M}^{\perp}(X) = \abs{X} \}.
  \end{equation*}
  We call $M^{\ast} \coloneqq (E, \mathcal{I}^{\ast})$ the \emph{dual} of $M$.
\end{dfn}

Note that, in general, $r_{M}^{\perp}$ need not coincide with the rank function $r_{M^{\ast}}$ of $M^{\ast}$.

\begin{prop} \label{prop:dual_of_independence_system}
  For any independence system $M$, the dual $M^{\ast}$ is
  an independence system.
\end{prop}
\begin{proof}
  First,
  \begin{equation*}
    r_{M}^{\perp}(\emptyset) = \abs{\emptyset} + r_{M}(E \setminus \emptyset) - r_{M}(E) = \abs{\emptyset},
  \end{equation*}
  so $\emptyset \in \mathcal{I}^{\ast}$.
  Now let $I_{1} \subseteq I_{2} \subseteq E$ with $I_{2} \in \mathcal{I}^{\ast}$.
  Since $r_{M}^{\perp}(I_{2}) = \abs{I_{2}}$, we have $r_{M}(E \setminus I_{2}) = r_{M}(E)$.
  By monotonicity of $r_{M}$,
  \begin{equation*}
    r_{M}(E) \geq r_{M}(E \setminus I_{1}) \geq r_{M}(E \setminus I_{2}) = r_{M}(E),
  \end{equation*}
  hence $r_{M}(E \setminus I_{1}) = r_{M}(E)$. Therefore
  \begin{equation*}
    r_{M}^{\perp}(I_{1}) = \abs{I_{1}} + r_{M}(E \setminus I_{1}) - r_{M}(E) = \abs{I_{1}},
  \end{equation*}
  so $I_{1} \in \mathcal{I}^{\ast}$. Thus $M^{\ast}$ is an independence system.
\end{proof}

\begin{dfn} \label{dfn:basis}
  Let $M = (E, \mathcal{I})$ be an independence system
  with the rank function $r_{M}$.
  We write
  \begin{equation*}
    \mathcal{B}(M) \coloneqq \{ B \in \mathcal{I} : \abs{B} = r_{M}(E) \}
  \end{equation*}
  and call the elements of $\mathcal{B}(M)$ the \emph{bases} of $M$.
\end{dfn}

\begin{rem}\label{rem:difference_of_dual}
  In some parts of combinatorial optimisation, one instead calls
  \emph{maximal} independent sets bases.
  One may then define a dual independence system by taking sets
  disjoint from such maximal independent sets,
  see, for instance, \cite[Section~13.3]{KorteVygen18}.
  For general independence systems, however, maximal independent sets
  need not have the same cardinality, so that construction is not adapted
  to rank complementation.
  We therefore use Definition~\ref{dfn:dual_of_independence_system},
  which is defined directly from the rank function and satisfies
  \begin{equation*}
    r_{M}^{\perp}(E) = \abs{E} - r_{M}(E).
  \end{equation*}
  When $M$ is a matroid, maximal and maximum independent sets coincide,
  and Definition~\ref{dfn:dual_of_independence_system} agrees with
  ordinary matroid duality.
\end{rem}

The following lemma gives a convenient description of $\mathcal{I}^{\ast}$
in terms of bases.

\begin{lem}\label{lem:Iast_via_bases}
  Let $M = (E, \mathcal{I})$ be an independence system with rank function $r_{M}$.
  Then
  \begin{equation*}
    \mathcal{I}^{\ast} = \{ I \subseteq E : \exists B \in \mathcal{B}(M) \text{ such that } I \subseteq E \setminus B \}
  \end{equation*}
\end{lem}
\begin{proof}
  By definition,
  \begin{equation*}
    I \in \mathcal{I}^{\ast}
    \iff r_{M}^{\perp}(I) = \abs{I}
    \iff r_{M}(E \setminus I) = r_{M}(E).
  \end{equation*}
  The latter holds if and only if $E \setminus I$ contains a basis of $M$,
  i.e., if and only if there exists $B \in \mathcal{B}(M)$ such that
  $B \subseteq E \setminus I$. Equivalently, $I \subseteq E \setminus B$
  for some $B \in \mathcal{B}(M)$.
\end{proof}

Borrowing terminology from abstract simplicial complexes,
we call an independence system $M$ \emph{pure}
if all maximal independent sets have the same cardinality.
For general independence systems, failure of purity is
exactly the obstruction to involutive duality.

\begin{prop} \label{prop:dual_is_involutive_iff_pure}
  Let $M = (E, \mathcal{I})$ be an independence system with rank function $r_{M}$,
  and let $r_{M^{\ast}}$ denote the rank function of $M^{\ast}$.
  Then, for every $X \subseteq E$,
  \begin{equation*}
    r_{M^{\ast}}(X) = \max_{B \in \mathcal{B}(M)} \abs{X \setminus B}
    = \abs{X} - \min_{B \in \mathcal{B}(M)} \abs{X \cap B}.
  \end{equation*}
  Moreover, the following are equivalent:
  \begin{enumerate}[label=\textup{(\arabic*)}, itemsep=0ex]
    \item $r_{M^{\ast}} = r_{M}^{\perp}$.
    \item $r_{M}(X) = \max_{B \in \mathcal{B}(M)} \abs{X \cap B}$ for all $X \subseteq E$.
    \item Every independent set of $M$ is contained in a basis of $M$.
    \item $M$ is pure.
  \end{enumerate}
  Under these equivalent conditions, one also has $M^{\ast\ast} = M$.
\end{prop}

\begin{proof}
  (1) $\Leftrightarrow$ (2):
  By Lemma~\ref{lem:Iast_via_bases}, an independent set $I \subseteq X$
  belongs to $I^{\ast}$ if and only if there exists $B \in \mathcal{B}(M)$
  such that $I \subseteq X \cap (E\setminus B)$. Hence
  \begin{equation*}
    r_{M^{\ast}}(X)
    = \max_{B \in \mathcal{B}(M)} \abs{X\cap (E\setminus B)}
    = \max_{B\in \mathcal{B}(M)} \abs{X \setminus B}
    = \abs{X} - \min_{B\in \mathcal{B}(M)} \abs{X \cap B}.
  \end{equation*}

  For each $B\in \mathcal{B}(M)$, we have
  \begin{equation*}
    \abs{(E\setminus X) \cap B} = r_{M}(E) - \abs{X \cap B}.
  \end{equation*}
  Therefore
  \begin{equation*}
    r_{M^{\ast}}(X) = r_{M}^{\perp}(X)
    \quad \text{for all } X \subseteq E
  \end{equation*}
  if and only if
  \begin{equation*}
    r_{M}(Y) = \max_{B\in \mathcal{B}(M)} \abs{Y\cap B}
    \quad \text{for all } Y \subseteq E,
  \end{equation*}
  and this proves (1) $\Leftrightarrow$ (2).

  (2) $\Leftrightarrow$ (3):
  Assume (2), and let $I \in \mathcal{I}$.
  Then $r_{M}(I) = \abs{I}$, so there exists $B\in \mathcal{B}(M)$
  such that $\abs{I} = r_{M}(I) = \abs{I \cap B}$.
  Hence $I\subseteq B$. Conversely, assume (3), and let $X\subseteq E$.
  Choose an independent set $I \subseteq X$ with $\abs{I} = r_{M}(X)$.
  By (3), there exists $B \in \mathcal{B}(M)$ such that $I \subseteq B$.
  Then
  \begin{equation*}
    r_{M}(X) = \abs{I} \leq \abs{X\cap B} \leq r_{M}(X),
  \end{equation*}
  which gives (2).

  (3) $\Leftrightarrow$ (4):
  Finally, since $E$ is finite, every independent set is contained in
  a maximal independent set. Therefore (3) holds if and only if
  every maximal independent set is a basis, namely if and only if
  $M$ is pure. Thus (3) $\Leftrightarrow$ (4).

  Under these equivalent conditions, the bases of $M^{\ast}$ are
  exactly the complements $E \setminus B$ with $B\in \mathcal{B}(M)$.
  Applying Lemma~\ref{lem:Iast_via_bases} to $M^{\ast}$,
  we obtain
  \begin{equation*}
    I^{\ast\ast}
    = \{ X \subseteq E : X \subseteq B \text{ for some } B \in \mathcal{B}(M)\}.
  \end{equation*}
  By (3), the right-hand side is exactly $I$. Hence $M^{\ast\ast} = M$.
\end{proof}

Thus, the obstruction to ordinary duality for a general independence system
is precisely the failure of purity (equivalently, the failure of basis extension).

As the following example shows, this duality need not hold in general;
moreover, even when $M(C)$ is a matroid,
the associated independence system $M(C^{\perp})$ may fail to be a matroid.

\begin{eg} \label{eg:duality_may_fail}
  Set $E \coloneqq \{ 1, 2, 3 \}$ and
  let $C \leq \mathbb{Z}_{4}^{E}$ be generated by
  \begin{equation*}
    G = \bordermatrix{
      & 1 & 2 & 3 \cr
      & 1 & 2 & 0 \cr
      & 0 & 2 & 2
    } \in (\mathbb{Z}_{4}^{2})^{E}.
  \end{equation*}
  Then, it is straightforward that $M(C) \cong U_{2, 3}$.
  On the other hand, 
  \begin{equation*}
  H = \bordermatrix{
      & 1 & 2 & 3 \cr
      & 2 & 1 & 1 \cr
      & 0 & 0 & 2
    } \in (\mathbb{Z}_{4}^{2})^{E}
  \end{equation*}
  is a parity-check matrix for $C$, but as we see in Example~\ref{eg:non_matroid},
  $M\lbrack H \rbrack = M(C^{\perp})$ is not a matroid.

  Set $M \coloneqq M(C^{\perp})$.
  Since $\mathcal{I} = \{ \emptyset, \{ 1 \}, \{ 2 \}, \{ 3 \}, \{ 2, 3 \} \}$,
  we get $\mathcal{I}^{\ast} = \{ \emptyset, \{ 1 \} \}$ by Definition~\ref{dfn:dual_of_independence_system}.
  However, taking the dual of $M^{\ast}$, $\mathcal{I}^{\ast\ast} = \{ \emptyset, \{ 2 \}, \{ 3 \}, \{ 2, 3 \} \}$,
  and thus $M^{\ast\ast} \neq M$.
\end{eg}

\begin{rem} \label{rem:duality_may_fail_even_if_free}
  We remark that, over a general local ring, the duality above can fail even for free codes.
  Set $E \coloneqq \{ 1, 2, 3, 4 \}$, and
  $R \coloneqq \mathbb{F}_{2} \lbrack x, y \rbrack/ \langle x^{2}, y^{2} \rangle$.
Then, the maximal ideal is $\mathfrak{m} = \langle x, y \rangle$.
  Let $C \leq R^{E}$ be the $R$-code generated by
  \begin{equation*}
    G = \bordermatrix{
      & 1 & 2 & 3 & 4 \cr
      & 1 & 0 & 0 & y \cr
      & 0 & 1 & x + y + xy & x + y
    } \in (R^{2})^{E}.
  \end{equation*}
  Since $G$ is in systematic form, $C$ is free.
  The matroid $M(C)$ is isomorphic to the single-element extension
  of $U_{2, 3}$ by adjoining a parallel element
  (elements $2$ and $3$ are parallel). On the other hand,
  \begin{equation*}
    H = \bordermatrix{
      & 1 & 2 & 3 & 4 \cr
      & 0 & x + y + xy & 1 & 0 \cr
      & y & x + y & 0 & 1
    } \in (R^{2})^{E}
  \end{equation*}
  is a parity-check matrix for $C$.
  However, the independence system $M\lbrack H \rbrack = M(C^{\perp})$
  violates the augmentation property (I3) with
  $I_{1} = \{ 2, 4 \}$ and $I_{2} = \{ 1, 2, 3 \}$.
  Indeed, $I_{1} \in \mathcal{I}$ is immediate.
  To see $I_{2} \in \mathcal{I}$, consider a linear relation
  among the columns of $H$ indexed by $I_{2}$:
  \begin{equation*}
    \alpha_{1} \begin{pmatrix} 0 \\ y \end{pmatrix} + \alpha_{2} \begin{pmatrix} x + y + xy \\ x + y \end{pmatrix} + \alpha_{3} \begin{pmatrix} 1 \\ 0 \end{pmatrix} = \begin{pmatrix} 0 \\ 0 \end{pmatrix}.
  \end{equation*}
  From the first coordinate, we obtain $\alpha_{3} = \alpha_{2}(x + y + xy) \in \mathfrak{m}$.
  If $\alpha_{2}$ were a unit, then $\alpha_{2}(x + y)$
  would have a nonzero $x$-term, whereas $\alpha_{1}y$ has no $x$-term;
  this contradicts the second coordinate equation
  $\alpha_{1}y + \alpha_{2}(x + y) = 0$.
  Hence $\alpha_{2} \in \mathfrak{m}$.
  It follows that $\alpha_{2}(x + y) \in \mathfrak{m}^{2} = \langle xy \rangle$,
  so the second coordinate implies $\alpha_{1} y \in \langle xy \rangle$,
  forcing the constant term of $\alpha_{1}$ to be zero, i.e.,
  $\alpha_{1} \in \mathfrak{m}$.
  Therefore $I_{2} \in \mathcal{I}$.
  The following two nontrivial linear relations
  \begin{align*}
    1 \cdot \begin{pmatrix} 0 \\ y \end{pmatrix} + 0 \begin{pmatrix} x + y + xy \\ x + y \end{pmatrix} + y \begin{pmatrix} 0 \\ 1 \end{pmatrix} &= \begin{pmatrix} 0 \\ 0 \end{pmatrix} \text{ and} \\
    1 \cdot \begin{pmatrix} x + y + xy \\ x + y \end{pmatrix} + (x + y + xy) \begin{pmatrix} 1 \\ 0 \end{pmatrix} + (x + y) \begin{pmatrix} 0 \\ 1 \end{pmatrix} &= \begin{pmatrix} 0 \\ 0 \end{pmatrix}
  \end{align*}
  show that $\{ 1, 2, 4 \}, \{ 2, 3, 4 \} \notin \mathcal{I}$.
  Consequently, no element of $I_{2} \setminus I_{1} = \{ 1, 3 \}$
  can be adjoined to $I_{1}$ while preserving independence,
  and thus $M\lbrack H \rbrack$ is not a matroid.
\end{rem}

However, if $R$ is a finite commutative chain ring,
we have $M(C)^{\ast} = M(C^{\perp})$ if $C$ is free.

\begin{lem} \label{lem:col_modular_indep_iff_row_modular_indep}
  For all $\ell \times \ell$ square matrices $A$,
  the columns of $A$ are modular independent if and only if
  the columns of $A^{\top}$ (or equivalently the rows of $A$)
  are modular independent.
\end{lem}
\begin{proof}
  Let $\mathfrak{m} = \langle \theta\rangle$, and
  let $\nu$ be the nilpotency index of $\mathfrak{m}$.
  Since $R$ is a PIR by Lemma~\ref{lem:noetherian_chain_iff_local_PIR},
  the matrix $A$ admits a Smith normal form. Thus there
  exist invertible matrices $P, Q\in R^{\ell \times \ell}$ such that
  \begin{equation*}
    PAQ = D \coloneqq \diag(d_{1}, \dots, d_{s}, 0, \dots, 0),
  \end{equation*}
  where $d_{i} \mid d_{i+1}$ for $1 \leq i < s$.

  By Remark~\ref{rem:modular_indep_with_syzygy},
  the columns of a square matrix $A$ are modular independent
  if and only if $\Ker A \subseteq \mathfrak{m}R^{\ell}$.
  Since $P$ and $Q$ are invertible, they induce automorphisms of $R^{\ell}$
  preserving $\mathfrak{m} R^{\ell}$.
  Hence
  \begin{equation*}
    \Ker A \subseteq \mathfrak{m}R^{\ell}
    \iff
    \Ker(PAQ)\subseteq \mathfrak{m}R^{\ell}.
  \end{equation*}

  We now examine the diagonal matrix $D$.
  If $D$ has a zero diagonal entry,
  then the corresponding standard basis vector belongs to
  $\Ker D\setminus \mathfrak{m}R^{\ell}$.
  Hence the columns of $D$, and therefore those of $A$,
  are modular dependent.

  Conversely, assume that $D$ has no zero diagonal entry.
  Then $s = \ell$, and each $d_{i}$ is non-zero.
  Since $R$ is a finite commutative chain ring,
  every non-zero element has the form
  $u_{i}\theta^{t_{i}}$ with $u_{i} \in R^{\times}$
  and $0 \leq t_{i} \leq \nu - 1$.
  Let $\bm{x} = (x_{1}, \dots, x_{\ell})^{\top} \in \Ker D$.
  Then $d_{i} x_{i} = 0$ for every $i \in \lbrack \ell \rbrack$, so
  if $u_{i} \theta^{t_{i}} x_{i} = 0$, then 
  $x_{i} \in \theta^{\nu - t_{i}}R \subseteq \mathfrak{m}$.
  Thus $\Ker D \subseteq \mathfrak{m}R^{\ell}$.
  Therefore the columns of $A$ are modular independent
  if and only if $D$ has no zero diagonal entry.

  Finally, transposing $PAQ=D$, we obtain
  \begin{equation*}
    Q^{\top} A^{\top} P^{\top} = D.
  \end{equation*}
  Hence $A$ and $A^{\top}$ have the same Smith normal form $D$.
  Therefore the columns of $A^{\top}$ are modular independent
  if and only if $D$ has no zero diagonal entry, equivalently
  if and only if the columns of $A$ are modular independent.
\end{proof}

The next lemma shows that, for matrices in systematic parity-check form
over a finite commutative chain ring, every independent set extends to a basis.

\begin{lem}\label{lem:extend_indep_in_H}
  Let $k^{\ast} \geq 1$ and $H = \lbrack B \mid I_{k^{\ast}} \rbrack \in R^{k^{\ast} \times n}$.
  Then every modular independent set of columns of $H$ is contained
  in a modular independent set of cardinality $k^{\ast}$.
  In particular, every independent set of $M \lbrack H \rbrack$
  is contained in some basis of $M \lbrack H \rbrack$.
\end{lem}

\begin{proof}
  Let $X \subseteq \lbrack n \rbrack$ be a modular independent set
  of columns of $H$, and set
  \begin{equation*}
    t \coloneqq \abs{X}
    \quad \text{and} \quad
    W \coloneqq \langle H_{X} \rangle_{R} \leq R^{k^{\ast}}.
  \end{equation*}
  By Lemma~\ref{lem:modular_indep_minimal_gen}, we have $\mu_{R}(W) = t$.
  If $t = k^{\ast}$, there is nothing to prove.
  Assume therefore that $t < k^{\ast}$.
  Let $\bm{e}_{1}, \dots, \bm{e}_{k^{\ast}}$ be the standard basis vectors
  of $R^{k^{\ast}}$, corresponding to the columns of the identity block in $H$.

  We claim that $W \cap \langle \bm{e}_{i} \rangle = 0$ for some $i \in \lbrack k^{\ast} \rbrack$.
  Suppose otherwise that $W \cap \langle \bm{e}_{i} \rangle \neq \{ \bm{0} \}$ for all $i \in \lbrack k^{\ast} \rbrack$.
  Set
  \begin{equation*}
    U_{i} \coloneqq W \cap \langle \bm{e}_{i} \rangle \leq \langle \bm{e}_{i} \rangle \cong R.
  \end{equation*}
  Since $R$ is a finite commutative chain ring, it is a PIR by Lemma~\ref{lem:noetherian_chain_iff_local_PIR}.
  Hence each nonzero $U_{i}$ is cyclic, and therefore $\mu_{R}(U_{i}) = 1$ for all $i \in \lbrack k^{\ast} \rbrack$.
  Moreover, because
  \begin{equation*}
    R^{k^{\ast}} = \bigoplus_{i = 1}^{k^{\ast}} \langle \bm{e}_{i} \rangle,
  \end{equation*}
  the sum $U_{1} + \dots + U_{k^{\ast}}$ is direct. Thus, by \eqref{eq:definition_of_muR},
  \begin{equation*}
    \mu_{R} \left( \bigoplus_{i = 1}^{k^{\ast}} U_{i} \right)
    = \sum_{i = 1}^{k^{\ast}} \mu_{R}(U_{i}) = k^{\ast}.
  \end{equation*}
  Since $\bigoplus_{i = 1}^{k^{\ast}} U_{i} \leq W$, Theorem~\ref{thm:monotonic_iff_chain_ring}
  yields
  \begin{equation*}
    k^{\ast} = \mu_{R} \left( \bigoplus_{i = 1}^{k^{\ast}} U_{i} \right)
    \leq \mu_{R}(W) = t,
  \end{equation*}
  contradicting $t < k^{\ast}$. This proves the claim.

  Choose $i \in \lbrack k^{\ast} \rbrack$ such that $W \cap \langle \bm{e}_{i} \rangle = \{ \bm{0} \}$,
  and let $j \in \lbrack n \rbrack$ be the index of the corresponding
  column of the identity block in $H$.
  Note that $j \notin X$ because $\bm{e}_{i} \notin W$.
  Applying Lemma~\ref{lem:muR_supermodular}
  to the submodules $W$ and $\langle \bm{e}_{i} \rangle$, we obtain
  \begin{equation*}
    \mu_{R}(W) + \mu_{R}(\langle \bm{e}_{i} \rangle)
    \leq \mu_{R}(W + \langle \bm{e}_{i} \rangle) + \mu_{R}(W \cap \langle \bm{e}_{i} \rangle).
  \end{equation*}
  Since $\mu_{R}(\langle \bm{e}_{i} \rangle) = 1$ and $W \cap \langle \bm{e}_{i} \rangle = \{ \bm{0} \}$,
  this gives
  \begin{equation*}
    t + 1 = \mu_{R}(W) + 1 \leq \mu_{R}(W + \langle \bm{e}_{i} \rangle).
  \end{equation*}
  On the other hand, $W + \langle \bm{e}_{i} \rangle = \langle H_{X \cup \{ j \}} \rangle_{R}$
  is generated by $t + 1$ columns, so $\mu_{R}(W + \langle \bm{e}_{i} \rangle) \leq t + 1$.
  Hence
  \begin{equation*}
    \mu_{R}(W + \langle \bm{e}_{i} \rangle) = t + 1.
  \end{equation*}
  By Lemma~\ref{lem:modular_indep_minimal_gen},
  the set $X \cup \{ j \}$ is modular independent.
  Repeating this argument, we obtain a modular independent superset
  $Y \supseteq X$ with $\abs{Y} = k^{\ast}$.

  Finally, the $k^{\ast}$ columns of the identity block are modular independent,
  so $r_{M\lbrack H \rbrack}(\lbrack n \rbrack) \geq k^{\ast}$.
  Conversely, if $Z \subseteq \lbrack n \rbrack$ is modular independent,
  then by Lemma~\ref{lem:modular_indep_minimal_gen} and Theorem~\ref{thm:monotonic_iff_chain_ring},
  \begin{equation*}
    \abs{Z} = \mu_{R}(\langle H_{Z} \rangle_{R}) \leq \mu_{R}(R^{k^{\ast}}) = k^{\ast}.
  \end{equation*}
  Therefore, $r_{M \lbrack H \rbrack}(\lbrack n \rbrack) = k^{\ast}$,
  and thus, every modular independent set of cardinality $k^{\ast}$
  is a basis of $M \lbrack H \rbrack$, and the final assertion follows.
\end{proof}

Remark~\ref{rem:duality_may_fail_even_if_free} also shows that
Lemma~\ref{lem:extend_indep_in_H} fails over a general local ring:
the matrix $H$ in Remark~\ref{rem:duality_may_fail_even_if_free}
has the form $\lbrack B \mid I_{2} \rbrack$,
yet $\{ 1, 2, 3 \}$ is modular independent.

\begin{thm} \label{thm:duality_for_free}
  If $C \leq R^{E}$ is a free $R$-code, then $M(C)^{\ast} = M(C^{\perp})$.
\end{thm}

\begin{proof}
  Since a simultaneous permutation of the coordinates merely relabels
  the ground set on both sides, we may identify $E$ with $\lbrack n \rbrack$
  and arrange that $C$ has a systematic generator matrix
  \begin{equation*}
    G = \lbrack I_{k} \mid A \rbrack,
  \end{equation*}
  where $k = \mu_{R}(C)$ and $n = \abs{E}$. 
  Then $H = \lbrack -A^{\top} \mid I_{n - k} \rbrack$
  is a parity-check matrix for $C$, so $M(C^{\perp}) = M\lbrack H \rbrack$.

  We first compare the bases of the two independence systems.
  Let $B \subseteq \lbrack n \rbrack$ with $\abs{B} = k$.
  Write
  \begin{equation*}
    B = P \sqcup (k + T), \quad \text{with }
    P \subseteq \lbrack k \rbrack \text{ and }
    T \subseteq \lbrack n - k \rbrack,
  \end{equation*}
  and set $S \coloneqq \lbrack k \rbrack \setminus P$.
  Then $\abs{S} = \abs{T}$.
  After suitable permutations of rows and columns,
  the submatrix $G_{B}$ is equivalent to
  \begin{equation*}
    \begin{pmatrix}
      I_{\abs{P}} & U \\
      0 & A_{S, T}
    \end{pmatrix}
  \end{equation*}
  for some matrix $U$. Then
  \begin{equation*}
    \Ker
      \begin{pmatrix}
      I_{\abs{P}} & U\\
      0 & A_{S,T}
      \end{pmatrix}
    =
    \left\{
      \begin{pmatrix} -Uy \\ y \end{pmatrix}
      \mathrel{}\middle|\mathrel{}
      y\in \Ker A_{S,T}
    \right\}.
  \end{equation*}
  In particular,
  \begin{equation*}
    \begin{pmatrix} -Uy \\ y \end{pmatrix} \in \mathfrak{m}^{\abs{P} + \abs{T}}
    \iff
    y\in \mathfrak{m}^{\abs{T}},
  \end{equation*}
  so Remark~\ref{rem:modular_indep_with_syzygy} shows that
  the columns of this block matrix are modular independent if and
  only if the columns of $A_{S,T}$ are modular independent.
  Hence $B$ is a basis of $M(C)$ if and only if the columns of
  the square matrix $A_{S, T}$ are modular independent.
  Likewise, the complement is
  \begin{equation*}
    \lbrack n \rbrack \setminus B = S \sqcup (k + (\lbrack n - k \rbrack \setminus T)),
  \end{equation*}
  and,  after suitable permutations of rows and columns,
  the corresponding submatrix of $H$ is equivalent to
  \begin{equation*}
    \begin{pmatrix}
      I_{(n-k) - \abs{T}} & U^{\prime}\\
      0 & -(A_{S,T})^{\top}
    \end{pmatrix}
  \end{equation*}
  for some matrix $U^{\prime}$, and the same kernel calculation
  shows that its columns are modular independent if and only if
  the columns of $(A_{S,T})^{\top}$ are modular independent.
  By Lemma~\ref{lem:col_modular_indep_iff_row_modular_indep},
  this is equivalent to modular independence of the columns of $A_{S, T}$.
  Thus,
  \begin{equation} \label{eq:basis_iff_complement_is_cobasis}
    B \in \mathcal{B}(M(C)) \iff \lbrack n \rbrack \setminus B \in \mathcal{B}(M(C^{\perp})).
  \end{equation}
  Now Lemma~\ref{lem:extend_indep_in_H} applied to $H = \lbrack -A^{\top} \mid I_{n - k} \rbrack$
  shows that every independent set of $M(C^{\perp}) = M\lbrack H \rbrack$
  is contained in a basis. Hence Proposition~\ref{prop:dual_is_involutive_iff_pure}
  also gives
  \begin{equation*}
    r_{M(C)^{\ast}}(X) = \max_{B \in \mathcal{B}(M(C))} \abs{X \setminus B}.
  \end{equation*}
  Therefore, by Lemma~\ref{lem:extend_indep_in_H} and Proposition~\ref{prop:dual_is_involutive_iff_pure}(2),
  \begin{align*}
    r_{M(C^{\perp})}(X)
    &= \max_{B^{\ast} \in \mathcal{B}(M(C^{\perp}))} \abs{X \cap B^{\ast}}
    = \max_{B \in \mathcal{B}(M(C))} \abs{X \cap (\lbrack n \rbrack \setminus B)} \\
    &= \max_{B \in \mathcal{B}(M(C))} \abs{X \setminus B}
    = r_{M(C)^{\ast}}(X).
  \end{align*}
  Hence $M(C^{\perp}) = M(C)^{\ast}$.
\end{proof} 
\section{Matroid Representation over Finite Chain Rings} \label{sect:application}

In this last section, we assume that $R$ is a finite commutative chain ring
with maximal ideal $\mathfrak{m} = \langle \theta \rangle$,
residue field $R/\mathfrak{m} = \mathbb{F}_{q}$, and
nilpotency index $\nu$.
We first define the terminology on representation.
As we have seen in Section~\ref{sect:matroid_repr_codes},
it would be convenient to represent a matroid by using a free $R$-code
because the duality behaves well in the free case.

\begin{dfn} \label{dfn:representability}
  Let $M$ be a matroid, $E$ be a finite set of cardinality $\abs{E(M)}$,
  and $V$ be an $R$-module.
  If $M$ is isomorphic to $M \lbrack A \rbrack$ for $A \in V^{E}$,
  we say $M$ is \emph{representable over $R$} or \emph{$R$-representable};
  and $A$ is a \emph{representation} for $M$ \emph{over} $R$
  or an \emph{$R$-representation} for $M$.
  Furthermore, if $V = R^{k}$ and, after permutations of columns,
  $A$ has the systematic form $\lbrack I_{k} \mid P \rbrack$,
  we say $M$ is \emph{freely representable over $R$} or
  \emph{freely $R$-representable};
  and $A$ is a \emph{free representation}
  or a \emph{free $R$-representation} for $M$.
\end{dfn}

We study the size constraints on representations of simple matroids.
For an $R$-module $V$, define
\begin{equation*}
  \Cyc(V) \coloneqq
  \{ \langle \bm{v} \rangle \mid \bm{v} \in V \}
\end{equation*}
partially ordered by inclusion.
Recall that the width of a partially ordered set $P$ is defined
as the supremum of the size of antichains (or clutters) in $P$:
\begin{equation*}
  \width(P) \coloneqq \sup \{ \abs{A} : A \subseteq P \text{ is an antichain} \}.
\end{equation*}

\begin{lem} \label{lem:size-constraint_simple}
  Let $V$ be an $R$-module,
  and $M$ be a simple matroid represented by $A \in V^{E}$.
  The map $E \to \Cyc(V), \, e \mapsto \langle A(e) \rangle$
  embeds $E$ into an antichain of $\Cyc(V)$.
  In particular, 
  \begin{equation*}
    \abs{E(M)} \leq \width(\Cyc(V)).
  \end{equation*}
\end{lem}

\begin{proof}
  Write $\bm{v}_{e} \coloneqq A(e)$ for each $e \in E$.
  Since $M$ is simple, it has no loops.
  Therefore $\bm{v}_{e} \neq \bm{0}$ for all $e \in E$.
  Also, $M$ has no parallel elements, so for distinct $e, f \in E$,
  the pair $\{ \bm{v}_{e}, \bm{v}_{f} \}$ is modular independent.
  By Corollary~\ref{cor:two_vectors_mod_dep},
  neither $\bm{v}_{e} \in \langle \bm{v}_{f} \rangle$ nor
  $\bm{v}_{f} \in \langle \bm{v}_{e} \rangle$ can occur.
  Hence $\langle \bm{v}_{e} \rangle$ and $\langle \bm{v}_{f} \rangle$
  are incomparable in $\Cyc(V)$.

  If $\langle \bm{v}_{e} \rangle = \langle \bm{v}_{f} \rangle$,
  then it contradicts modular independence of $\{ \bm{v}_{e}, \bm{v}_{f} \}$.
  Thus, $e \mapsto \langle \bm{v}_{e} \rangle$ is injective,
  and its image is an antichain; hence, $\abs{E(M)} \leq \width(\Cyc(V))$.
\end{proof}

Next, we study the representation of uniform matroids.

\begin{prop} \label{prop:uniform_matroid_freely_representable}
  The uniform matroid $U_{k, n}$ is freely $R$-representable
  if it is $R$-representable.
\end{prop}

\begin{proof}
  Let $A \colon V^{E}$ be a (not necessarily free) representation of $U_{k, n}$.
  Since $R$ is a chain ring, Corollary~\ref{cor:rank_when_chain} gives
  \begin{equation*}
    \mu_{R}(V_{E}) = r_{M \lbrack A \rbrack}(E) = k.
  \end{equation*}
  Choose a minimal subset $B = \{ b_{1}, \dots, b_{k} \} \subseteq E$
  such that $A(B)$ generates $V_{E}$.
  By Theorem~\ref{thm:number_of_minimal_generators},
  $\abs{B} = \mu_{R}(V_{E}) = k$.
  Define a surjection
  \begin{equation*}
    \varphi \colon R^{k} \twoheadrightarrow V_{E}, \;
    \bm{e}_{i} \mapsto A(b_{i}) \quad (i \in \lbrack k \rbrack),
  \end{equation*}
  where $\bm{e}_{1}, \dots, \bm{e}_{k}$ are the standard basis vectors of $R^{k}$.
  For each $e \in E$, choose $\bm{g}_{e} \in R^{k}$ with
  $\varphi(\bm{g}_{e}) = A(e)$, and set $G \coloneqq (\bm{g}_{e})_{e \in E} \in (R^{k})^{E}$.
  After relabelling $E$, we may assume $B = \lbrack k \rbrack$, so that 
  $\bm{g}_{b_{i}} = \bm{e}_{i}$ for $i \in \lbrack k \rbrack$.
  Hence, $G = \lbrack I_{k} \mid P \rbrack$ for some matrix $P$,
  and therefore the row span $C \leq R^{E}$ of $G$ is a free $R$-code
  of rank $k$. In particular, Proposition~\ref{prop:systematic_form_for_free_codes}
  applies to $C$.

  We claim that $M(C) = M\lbrack G \rbrack \cong U_{k, n}$.
  Let $I \subseteq E$ with $\abs{I} \leq k$.
  If the columns of $G_{I}$ were modular dependent,
  then applying $\varphi$ to a dependence relation would give
  a modular dependence among the columns of $A_{I}$,
  contradicting the fact that every subset of size at most $k$
  is independent in $M\lbrack A \rbrack \cong U_{k, n}$.
  Thus every subset of at most $k$ columns of $G$ is modular independent.

  Conversely, if $D \subseteq E$ satisfies $\abs{D} > k$, then
  \begin{equation*}
    \mu_{R}(\langle G(D) \rangle_{R}) \leq \mu_{R}(R^{k}) = k
  \end{equation*}
  by Theorem~\ref{thm:monotonic_iff_chain_ring}.
  Hence Lemma~\ref{lem:modular_indep_minimal_gen} implies that
  the columns of $G_{D}$ are modular dependent.
  Therefore $M(C) = M \lbrack G \rbrack \cong U_{k, n}$.
\end{proof}

\begin{lem} \label{lem:free_representation_size_constraint}
  Let $C \leq R^{E}$ be a free $R$-code with $\mu_{R}(C) = k$.
  If $M(C)$ is a simple matroid, then
  \begin{equation*}
    \abs{E} \leq
    \frac{q^{\nu k} - q^{(\nu-1) k}}{q^{\nu} - q^{\nu-1}}
    =
    q^{(\nu - 1)(k - 1)}\frac{q^{k} - 1}{q - 1}.
  \end{equation*}
\end{lem}

\begin{proof}
  Let $G = (\bm{g}_{e})_{e\in E} \in (R^{k})^{E}$ be
  a generator matrix for $C$.
  Since $M(C)$ is simple, it has no loops, and hence $\bm{g}_{e} \neq 0$
  for every $e\in E$. For each $e\in E$,
  let $m_{e}$ be the largest integer such that $\bm{g}_{e}\in \theta^{m_{e}}R^{k}$,
  and write $\bm{g}_{e} = \theta^{m_{e}} \bm{u}_{e}$
  with $\bm{u}_{e} \notin \theta R^{k}$. Then $\bm{u}_{e}$ is primitive.

  We claim that the cyclic submodules $\langle \bm{u}_{e} \rangle$
  are pairwise distinct.
  Suppose $\langle \bm{u}_{e} \rangle = \langle \bm{u}_{f} \rangle$
  for some distinct $e, f \in E$.
  Then $\bm{u}_{f} = \varepsilon \bm{u}_{e}$ for some unit $\varepsilon \in R^\times$.
  Since $R$ is a chain ring, the ideals $\langle \theta^{m_{e}}\rangle$
  and $\langle \theta^{m_{f}}\rangle$ are comparable;
  without loss of generality, assume $m_{e} \leq m_{f}$. Then
  \begin{equation*}
    \bm{g}_{f} = \theta^{m_{f}}\bm{u}_{f}
    = \varepsilon \theta^{m_{f} - m_{e}} \bm{g}_{e}
    \in \langle \bm{g}_{e} \rangle.
  \end{equation*}
  Hence $\{ \bm{g}_{e}, \bm{g}_{f} \}$ is modular dependent
  by Corollary~\ref{cor:two_vectors_mod_dep}.
  This means that $\{ e, f \}$ is dependent in $M(C)$,
  contradicting simplicity.

  Therefore the submodules $\langle \bm{u}_{e} \rangle$ are
  pairwise distinct. By Corollary~\ref{cor:counting_formula},
  the number of cyclic submodules of $R^{k}$ generated
  by primitive vectors is
  \begin{equation*}
    \frac{q^{\nu k}-q^{(\nu - 1)k}}{q^{\nu} - q^{\nu-1}}
    =
    q^{(\nu - 1)(k - 1)}\frac{q^{k} - 1}{q - 1}.
  \end{equation*}
  Hence the desired inequality follows.
\end{proof}

The following theorem shows that representability over a finite chain ring
can be substantially richer than representability over a field
of the same cardinality.

\begin{thm} \label{thm:excluded_minor_uniform}
  The uniform matroid $U_{2, n}$ is representable over $R$
  if and only if $n \leq q^{\nu} + q^{\nu - 1}$.
  In particular, neither $U_{2,q^{\nu} + q^{\nu - 1} + 1}$ nor
  $U_{q^{\nu} + q^{\nu - 1} - 1, q^{\nu} + q^{\nu - 1} + 1}$
  is representable over $R$.
\end{thm}

\begin{proof}
  Suppose that $U_{2,n}$ is representable over $R$.
  By Proposition~\ref{prop:uniform_matroid_freely_representable},
  there exists a free $R$-code $C \leq R^{E}$ with $\mu_{R}(C) = 2$
  such that
  \begin{equation*}
    M(C) \cong U_{2, n}.
  \end{equation*}
  Since $U_{2, n}$ is simple, Lemma~\ref{lem:free_representation_size_constraint}
  gives $n \leq q^{\nu} + q^{\nu - 1} \eqqcolon N$.

  Conversely, assume that $n \leq N$. Consider the set
  \begin{equation*}
    \mathbb{P}^{1}(R) \coloneqq
    \{ (1,r)^\top \mid r\in R \} \cup \{(l,1)^\top \mid l\in \mathfrak{m}\}
    \subseteq R^{2}
  \end{equation*}
  a standard set of representatives of the projective line over $R$ (see, e.g., \cite[Theorem~3.5.5]{Havlicek12}).
  Then $\abs{\mathbb{P}^{1}(R)} = \abs{R} + \abs{\mathfrak{m}} = q^{\nu} + q^{\nu - 1} = N$.
  Let $\bm{u}, \bm{v} \in \mathbb{P}^{1}(R)$ be distinct.
  \begin{itemize}[itemsep=0ex]
    \item If $\bm{u} = (1,r)^{\top}$ and $\bm{v} = (1, r^{\prime})^{\top}$,
      then $\det(\bm{u}, \bm{v}) = r^{\prime} - r \neq 0$.
    \item If $\bm{u} = (l,1)^{\top}$ and $\bm{v} = (l^{\prime}, 1)^{\top}$,
      then $\det(\bm{u}, \bm{v}) = l - l^{\prime} \neq 0$.
    \item If $\bm{u} = (1,r)^{\top}$ and $\bm{v} = (l,1)^{\top}$ with $l \in \mathfrak{m}$,
      then $\det(\bm{u}, \bm{v}) = 1 - lr \in R^{\times}$.
  \end{itemize}
  Thus every two-element subset of $\mathbb{P}^{1}(R)$ is
  modular independent by Proposition~\ref{prop:nonzero_minor_implies_modular_indep}.

  If $X \subseteq \mathbb{P}^{1}(R)$ is modular independent,
  then Lemma~\ref{lem:modular_indep_minimal_gen} and
  Theorem~\ref{thm:monotonic_iff_chain_ring} give
  \begin{equation*}
    \abs{X} = \mu_{R}(\langle X \rangle_{R}) \leq \mu_{R}(R^{2}) = 2.
  \end{equation*}
  Hence every three-element subset of $\mathbb{P}^{1}(R)$ is modular dependent.
  Therefore any $n$-element subset of $\mathbb{P}^{1}(R)$ represents
  $U_{2, n}$.

  The non-representability of $U_{2,N+1}$ is immediate from the first assertion.
  If $U_{N-1,N+1}$ were representable over $R$, then Proposition~\ref{prop:uniform_matroid_freely_representable}
  would yield a free $R$-code $C \leq R^{N+1}$ such that $M(C)\cong U_{N-1,N+1}$.
  By Theorem~\ref{thm:duality_for_free},
  $M(C^{\perp}) = M(C)^{\ast} \cong U_{2,N+1}$,
  contradicting the first assertion.
\end{proof}

Note that, if $R = \mathbb{F}_{q}$ is a finite field, then $\nu = 1$;
and so Theorem~\ref{thm:excluded_minor_uniform} reduces
to the classical fact that $U_{2,q+2}$ and $U_{q, q+2}$ are
not $\mathbb{F}_{q}$-representable.

\begin{eg} \label{eg:representation_of_uniform_matroids_of_rank2}
  \begin{equation*}
    \begin{pmatrix}
      1 & 0 & 1 & 1 & 1 & 2 \\
      0 & 1 & 1 & 2 & 3 & 1
    \end{pmatrix} \in \mathbb{Z}_{4}^{2 \times 6}
  \end{equation*}
  gives a representation of $U_{2, 6}$ over $\mathbb{Z}_{4}$;
  in particular, $U_{2, 6}$ is not $\mathbb{F}_{4}$-representable
  but $\mathbb{Z}_{4}$-representable.
  By Theorem~\ref{thm:excluded_minor_uniform}, 
  $U_{2, 7}$ and $U_{5, 7}$ are not $\mathbb{Z}_{4}$-representable.
  \begin{equation*}
    \left( \begin{array}{cccccccccccc}
      1 & 0 & 1 & 1 & 1 & 1 & 1 & 1 & 1 & 2 & 4 & 6 \\
      0 & 1 & 1 & 2 & 3 & 4 & 5 & 6 & 7 & 1 & 1 & 1
    \end{array} \right) \in \mathbb{Z}_{8}^{2 \times 12}
  \end{equation*}
  gives a representation of $U_{2, 12}$ over $\mathbb{Z}_{8}$;
  in particular, $U_{2, 12}$ is not $\mathbb{F}_{8}$-representable
  but $\mathbb{Z}_{8}$-representable.
  By Theorem~\ref{thm:excluded_minor_uniform}, 
  $U_{2, 13}$ and $U_{11, 13}$ are not $\mathbb{Z}_{8}$-representable.
\end{eg}

We now study the $\mathbb{Z}_{4}$-representability,
comparing with the $\mathbb{F}_{4}$-representability.
$\mathbb{F}_{4}$-representable matroids are characterised
by seven excluded minors in an important paper of
Geelen, Gerards, and Kapoor \cite{ExcludedMinorGF4};
see also \cite[Section~6.5]{Oxley}.
The previous examples already show that $\mathbb{Z}_{4}$-
and $\mathbb{F}_{4}$-representability diverge on uniform matroids.
The next proposition shows that this difference is also visible from
the excluded-minor viewpoint: none of the seven excluded minors
for $\mathbb{F}_{4}$-representability remains an obstruction over $\mathbb{Z}_{4}$.

\begin{prop} \label{prop:all_excluded_minors_GF4_are_Z4_representable}
  All excluded minors for the class of
  $\mathbb{F}_{4}$-representable matroids,
  namely $U_{2, 6}$, $U_{4, 6}$, $P_{6}$, $F_{7}^{-}$,
  $(F_{7}^{-})^{\ast}$, $P_{8}$ and $P_{8}^{=}$,
  are representable over $\mathbb{Z}_{4}$.
\end{prop}

\begin{proof}
  A $\mathbb{Z}_{4}$-representation of $U_{2,6}$ is given in
  Example~\ref{eg:representation_of_uniform_matroids_of_rank2},
  and $\mathbb{Z}_{4}$-representations of $P_{6}$, $F_{7}^{-}$, $P_{8}$, and $P_{8}^{=}$
  are listed in Appendix~\ref{sect:appendix}.
  Since each displayed matrix is in systematic form
  after a suitable permutation of columns,
  Theorem~\ref{thm:duality_for_free} yields $\mathbb{Z}_{4}$-representations
  of the dual matroids $U_{4, 6} = U_{2, 6}^{\ast}$ and $(F_{7}^{-})^{\ast}$.
\end{proof}

Thus, Proposition~\ref{prop:all_excluded_minors_GF4_are_Z4_representable}
should be read as a comparison with the $\mathbb{F}_{4}$-theory,
rather than as an excluded minor statement for $\mathbb{Z}_{4}$-representability.
Moreover, we have the following proposition from Appendix~\ref{sect:appendix}.

\begin{prop} \label{prop:non-representable}
  Non-representable matroids $F_{8}$ and $\AG(3, 2)^{\prime}$
  are freely $\mathbb{Z}_{4}$-representable.
\end{prop}

The matrices in Example~\ref{eg:representation_of_uniform_matroids_of_rank2} and Appendix~\ref{sect:appendix}
are intended as explicit witnesses of representability.
Verifying that a displayed matrix realises the claimed matroid
amounts to a routine check of modular (in)dependence
on small subsets of columns; we omit these checks in the interest of space.
As a representative example, we include one verification in full:
Theorem~\ref{thm:Vamos_over_Z8} proves that the well-known smallest
non-representable matroid $V_{8}$ is freely $\mathbb{Z}_{8}$-representable.

\begin{thm} \label{thm:Vamos_over_Z8}
  The V\'{a}mos matroid $V_{8}$ is freely $\mathbb{Z}_{8}$-representable.
\end{thm}

\begin{proof}
  Set $E \coloneqq \{ a, b, c, d, e, f, g, h \}$.
  Throughout this proof, we use the standard labelling of the V\'{a}mos matroid $V_{8}$ on $E$
  in which the $4$-circuits are
  \begin{equation} \label{eq:4circuits_of_V8} 
    \{ a, b, c, d \},\; \{ a, b, e, f \},\; \{ c, d, e, f \},\;
    \{ a, b, g, h \},\; \{ c, d, g, h \},
  \end{equation}
  and the $5$-circuits are precisely the $5$-subsets $D \subseteq E$
  that contain none of these five $4$-sets.
  Let $G \in (\mathbb{Z}_{8}^{4})^{E}$ be the following matrix:
  \begin{table}[H]
    \begin{tabular}{cc}
      \begin{minipage}{.6\hsize}
        \begin{equation*}
          G \coloneqq \bordermatrix{
            & a & b & c & d & e & f & g & h \cr
            & 1 & 0 & 0 & 1 & 0 & 1 & 2 & 1 \cr
            & 0 & 1 & 0 & 1 & 0 & 1 & 3 & 2 \cr
            & 0 & 0 & 1 & 1 & 0 & 0 & 1 & 7 \cr
            & 0 & 0 & 0 & 0 & 1 & 1 & 4 & 4 
          } \in (\mathbb{Z}_{8}^{4})^{E}
        \end{equation*}
      \end{minipage}
      &
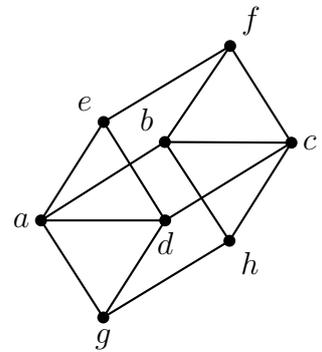
      \begin{minipage}{.3\hsize}
        \centering
        \begin{tikzpicture}[scale=0.45, thick, line cap=round, line join=round]
          \coordinate (a) at (0.00,  0.00);
          \coordinate (d) at (3.63,  0.00);
          \coordinate (b) at (3.62,  2.31);
          \coordinate (e) at (1.83,  2.90);
          \coordinate (f) at (5.53,  5.14);
          \coordinate (c) at (7.32,  2.29);
          \coordinate (h) at (5.51, -0.60);
          \coordinate (g) at (1.82, -2.86);

          \foreach \U/\V in {a/e,e/f,f/c,c/h,h/g,g/a,a/d,d/g,e/d,a/b,b/c,b/h,c/d,f/b}{
            \draw (\U) -- (\V);
          }

          \fill (a) circle (5pt) node[left]        {$a$};
          \fill (e) circle (5pt) node[above left]  {$e$};
          \fill (b) circle (5pt) node[above left]  {$b$};
          \fill (c) circle (5pt) node[right]       {$c$};
          \fill (d) circle (5pt) node[below]       {$d$};
          \fill (f) circle (5pt) node[above right] {$f$};
          \fill (g) circle (5pt) node[below]       {$g$};
          \fill (h) circle (5pt) node[below right] {$h$};
        \end{tikzpicture}
      \end{minipage}
    \end{tabular}
  \end{table}
  We verify that the associated independence system $M\lbrack G \rbrack$
  is $V_{8}$.

  First, the submatrix $G_{\{ a, b, c, e \}}$ is the $4\times 4$
  identity matrix, so $\{ a, b, c, e \}$ is modular independent.
  On the other hand, since $\mathbb{Z}_{8}$ is a commutative chain ring,
  any modular independent subset of columns in $\mathbb{Z}_{8}^{4}$
  has cardinality at most $4$ (Lemma~\ref{lem:modular_indep_minimal_gen}
  and Theorem~\ref{thm:monotonic_iff_chain_ring}).
  Hence $r(M\lbrack G \rbrack) = 4$.
  In particular, after permuting the columns so that $\{ a, b, c, e\}$
  forms the identity block, the matrix is in systematic form and
  thus gives a \emph{free} $\mathbb{Z}_{8}$-representation.

  Next, we claim that $M\lbrack G \rbrack$ has no circuits of size at most $3$.
  Indeed, a direct check shows that for every $3$-subset $X \subseteq E$,
  the $4 \times 3$ submatrix $G_{X}$ has a non-zero $3 \times 3$ minor;
  therefore $X$ is modular independent by
  Proposition~\ref{prop:nonzero_minor_implies_modular_indep}.

  We now determine the $4$-circuits.
  Write $\bm{g}_{i} \coloneqq G(i)$ for each $i \in E$.
  In $R^{4}$, the following relations hold:
  \begin{gather*}
    \bm{g}_{d} = \bm{g}_{a} + \bm{g}_{b} + \bm{g}_{c}, \qquad
    \bm{g}_{f} = \bm{g}_{a} + \bm{g}_{b} + \bm{g}_{e}, \qquad
    \bm{g}_{c} + \bm{g}_{f} = \bm{g}_{d} + \bm{g}_{e}, \\
    \bm{g}_{a} + 5\bm{g}_{g} + 5 \bm{g}_{h} = \bm{g}_{b}, \qquad
    \bm{g}_{c} + \bm{g}_{d} + \bm{g}_{h} = \bm{g}_{g}.
  \end{gather*}
  Since the displayed coefficients are units of $\mathbb{Z}_{8}$, each of
  \begin{equation*}
    \{ a, b, c, d \},\; \{ a, b, e, f \},\; \{ c, d, e, f \},\; \{ a, b, g, h \},\; \{ c, d, g, h \}
  \end{equation*}
  is modular dependent (Lemma~\ref{lem:modular_dep_lin_comb}),
  and by the previous paragraph it is a $4$-circuit.

  On the other hand, for every other $4$-subset $X\subseteq E$,
  a direct computation shows that $\det(G_{X})\neq 0$ over $\mathbb{Z}_{8}$.
  Hence $X$ is modular independent by
  Proposition~\ref{prop:nonzero_minor_implies_modular_indep}.
  Therefore, the $4$-circuits of $M\lbrack G \rbrack$ are \emph{exactly}
  $\{ a, b, c, d \}$, $\{ a, b, e, f \}$, $\{ c, d, e, f \}$, 
  $\{ a, b, g, h \}$, $\{ c, d, g, h \}$.

  Finally, let $D\subseteq E$ with $\abs{D} = 5$.
  If $D$ contains one of the above $4$-circuits, then $D$ is dependent
  but not minimal, hence not a circuit.
  If $D$ contains none of them, then every $4$-subset of $D$
  is independent, while $D$ itself is dependent because $r(M\lbrack G \rbrack) = 4$.
  Hence such $D$ is a $5$-circuit.
  Therefore the $5$-circuits of $M\lbrack G \rbrack$ are precisely
  the $5$-subsets of $E$ containing none of
  $\{ a, b, c, d \}$, $\{ a, b, e, f \}$, $\{ c, d, e, f \}$,
  $\{ a, b, g, h \}$, $\{ c, d, g, h \}$, exactly as in \eqref{eq:4circuits_of_V8}.

  We conclude that $M\lbrack G \rbrack = V_{8}$,
  and thus $V_{8}$ is freely $\mathbb{Z}_{8}$-representable.
\end{proof} 
\appendix
\section{Appendix} \label{sect:appendix}

In this appendix we list explicit representations of
several well-known matroids; see \cite[Appendix]{Oxley}
for further details on each of them.

\subsection*{$P_{6}$}

$\mathbb{F}$-representable if and only if $\abs{\mathbb{F}} \geq 5$.
An excluded minor for $\mathbb{F}_{4}$-representability.

\begin{table}[H]
  \begin{tabular}{cc}
    \begin{minipage}{.45\hsize}
      \centering
      \begin{tikzpicture}[auto]
        \fill (-1.5, 0)     coordinate (a) circle (3pt) node [below=2mm]   {$a$};
        \fill ( 0, 0)       coordinate (b) circle (3pt) node [below=2mm]   {$b$};
        \fill ( 1.5, 0)     coordinate (c) circle (3pt) node [below=2mm]   {$c$};
        \fill (-1.06, 1.06) coordinate (d) circle (3pt) node [above left]  {$d$};
        \fill ( 0, 1.5)     coordinate (e) circle (3pt) node [above=1mm]   {$e$};
        \fill ( 1.06, 1.06) coordinate (f) circle (3pt) node [above right] {$f$};
        \draw (a) -- (c);
      \end{tikzpicture}
    \end{minipage}
    &
    \begin{minipage}{.45\hsize}
      \begin{equation*}
        \bordermatrix{
          & a & b & c & d & e & f \cr
          & 1 & 0 & 1 & 0 & 1 & 3 \cr
          & 0 & 1 & 1 & 0 & 2 & 1 \cr
          & 0 & 0 & 0 & 1 & 1 & 1
        } \quad \text{over } \mathbb{Z}_{4}
      \end{equation*}
    \end{minipage}
  \end{tabular}
\end{table}

\subsection*{$F_{7}^{-}$}

$\mathbb{F}$-representable if and only if the characteristic of $\mathbb{F}$ is not two.
An excluded minor for $\mathbb{F}_{4}$-representability.

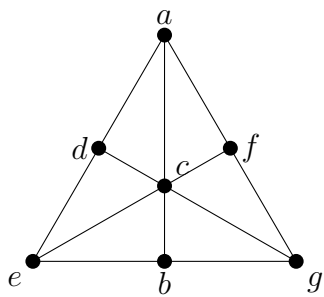
\begin{table}[H]
  \begin{tabular}{cc}
    \begin{minipage}{.45\hsize}
      \centering
      \begin{tikzpicture}[scale=2]
        \coordinate (A) at ( 90:1);
        \coordinate (L) at (210:1);
        \coordinate (R) at (330:1);
        \coordinate (LR) at ($(L)!0.5!(R)$);
        \coordinate (AR) at ($(R)!0.5!(A)$);
        \coordinate (AL) at ($(A)!0.5!(L)$);
        \coordinate (O) at (0,0);           

        \draw (A)--(L)--(R)--cycle;
        \draw (A)--(LR);
        \draw (L)--(AR);
        \draw (R)--(AL);

        \fill (A) circle (1.4pt) node[above] {$a$};
        \fill (LR) circle (1.4pt) node[below] {$b$};
        \fill (O) circle (1.4pt) node[above right] {$c$};
        \fill (AL) circle (1.4pt) node[left] {$d$};
        \fill (L) circle (1.4pt) node[below left] {$e$};
        \fill (AR) circle (1.4pt) node[right] {$f$};
        \fill (R) circle (1.4pt) node[below right] {$g$};
      \end{tikzpicture}
    \end{minipage}
    &
    \begin{minipage}{.45\hsize}
      \begin{equation*}
        \bordermatrix{
          & a & b & c & d & e & f & g \cr
          & 1 & 0 & 1 & 0 & 1 & 0 & 1 \cr
          & 0 & 1 & 1 & 0 & 0 & 1 & 1 \cr
          & 0 & 0 & 0 & 1 & 1 & 3 & 1  
        } \quad \text{over } \mathbb{Z}_{4}
      \end{equation*}
    \end{minipage}
  \end{tabular}
\end{table}

\subsection*{$P_{8}^{=}$}

$\mathbb{F}$-representable if and only if $\abs{\mathbb{F}} \geq 5$.
An excluded minor for $\mathbb{F}_{4}$-representability.

\begin{table}[H]
  \begin{tabular}{cc}
    \begin{minipage}{.45\hsize}
      \begin{equation*}
        \bordermatrix{
          & a & b & d & e & g & h & f & c \cr
          & 1 & 0 & 0 & 0 & 1 & 1 & 1 & 1 \cr
          & 0 & 1 & 0 & 0 & 1 & 1 & 4 & 2 \cr
          & 0 & 0 & 1 & 0 & 1 & 2 & 0 & 2 \cr
          & 0 & 0 & 0 & 1 & 1 & 4 & 1 & 0
        } \quad \text{over } \mathbb{F}_{5},
      \end{equation*}
    \end{minipage}
    &
    \begin{minipage}{.45\hsize}
      \begin{equation*}
        \bordermatrix{
          & a & b & c & d & e & f & g & h \cr
          & 1 & 0 & 0 & 1 & 0 & 2 & 1 & 0 \cr
          & 0 & 1 & 0 & 2 & 0 & 1 & 0 & 3 \cr
          & 0 & 0 & 1 & 1 & 0 & 0 & 3 & 1 \cr
          & 0 & 0 & 0 & 0 & 1 & 1 & 1 & 1
        } \quad \text{over } \mathbb{Z}_{4}
      \end{equation*}
    \end{minipage}
  \end{tabular}
\end{table}

\subsection*{$P_{8}$}

$\mathbb{F}$-representable if and only if the characteristic of $\mathbb{F}$ is not two.
An excluded minor for $\mathbb{F}_{4}$-representability.

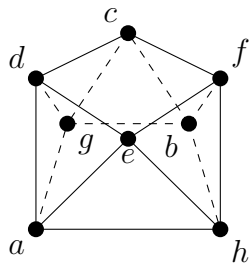
\begin{table}[H]
  \begin{tabular}{cc}
    \begin{minipage}{.45\hsize}
      \centering
      \begin{tikzpicture}[auto]
        \fill (-1.213,-1.2) coordinate (a) circle (3pt) node [below left]  {$a$}; \fill ( 0.8, 0.2)   coordinate (b) circle (3pt) node [below left]  {$b$}; \fill (0, 1.4)      coordinate (c) circle (3pt) node [above left]  {$c$}; \fill (-1.213, 0.8) coordinate (d) circle (3pt) node [above left]  {$d$}; \fill (0, 0)        coordinate (e) circle (3pt) node [below]       {$e$}; \fill ( 1.213, 0.8) coordinate (f) circle (3pt) node [above right] {$f$}; \fill (-0.8, 0.2)   coordinate (g) circle (3pt) node [below right] {$g$}; \fill ( 1.213,-1.2) coordinate (h) circle (3pt) node [below right] {$h$}; 

        \draw (a) to (d);
        \draw (a) to (e);
        \draw[dashed] (a) to (g);
        \draw (a) to (h);
        \draw[dashed] (b) to (c);
        \draw[dashed] (b) to (f);
        \draw[dashed] (b) to (g);
        \draw[dashed] (b) to (h);
        \draw (c) to (d);
        \draw (c) to (f);
        \draw[dashed] (c) to (g);
        \draw (d) to (e);
        \draw[dashed] (d) to (g);
        \draw (e) to (f);
        \draw (e) to (h);
        \draw (f) to (h);
      \end{tikzpicture}
    \end{minipage}
    &
    \begin{minipage}{.45\hsize}
      \begin{equation*}
        \bordermatrix{
          & a & b & c & d & e & f & g & h \cr
          & 1 & 1 & 0 & 0 & 0 & 0 & 3 & 1 \cr
          & 0 & 2 & 0 & 1 & 0 & 1 & 0 & 2 \cr
          & 0 & 0 & 1 & 1 & 0 & 0 & 1 & 1 \cr
          & 0 & 0 & 0 & 0 & 1 & 1 & 1 & 1 
        } \quad \text{over } \mathbb{Z}_{4}
      \end{equation*}
    \end{minipage}
  \end{tabular}
\end{table}

\subsection*{$F_{8}$}

A smallest non-representable and non-algebraic matroid.

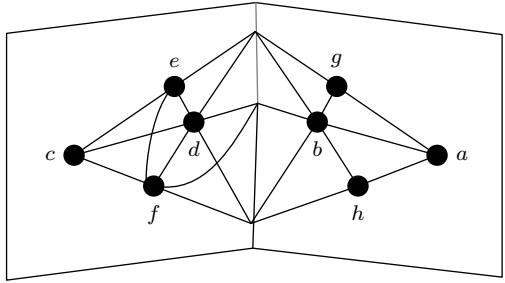
\begin{table}[H]
  \begin{tabular}{cc}
    \begin{minipage}{.45\hsize}
      \centering
      \begin{tikzpicture}[
          scale=0.8,
          x=1cm,y=1cm,
          line cap=round,line join=round,
          line width=0.5pt
        ]
        \coordinate (BL)  at (0.00,0.00);
        \coordinate (TL)  at (0.00,4.09);
        \coordinate (TOP) at (4.10,4.60);
        \coordinate (TR)  at (8.15,4.07);
        \coordinate (BR)  at (8.15,0.05);
        \coordinate (BOT) at (4.05,0.53);

        \coordinate (SpTop) at (4.09,4.12);
        \coordinate (SpMid) at (4.13,2.93);
        \coordinate (SpBot) at (4.02,0.94);

        \coordinate (LO) at (1.11,2.07);
        \coordinate (LU) at (2.76,3.21);
        \coordinate (LI) at (3.08,2.62);
        \coordinate (LL) at (2.42,1.56);

        \coordinate (RU) at (5.43,3.21);
        \coordinate (RI) at (5.11,2.62);
        \coordinate (RL) at (5.78,1.56);
        \coordinate (RO) at (7.08,2.07);

        \draw (TL)--(TOP)--(TR)--(BR)--(BOT)--(BL)--cycle;
        \draw[gray] (TOP)--(SpMid);
        \draw      (SpMid)--(BOT);

        \draw (LO)--(LU);
        \draw (LO)--(LI);
        \draw (LO)--(LL);
        \draw (LU)--(LI);
        \draw (LI)--(LL);
        \draw (LU) to[out=-140,in=-110] (LL); \draw (LL) to[out=-10,in=-120] (SpMid); 

        \draw (RO)--(RU);
        \draw (RO)--(RI);
        \draw (RO)--(RL);
        \draw (RU)--(RI);
        \draw (RI)--(RL);

        \draw (LU)--(SpTop);
        \draw (LI)--(SpTop);
        \draw (RU)--(SpTop);
        \draw (RI)--(SpTop);

        \draw (LI)--(SpMid);
        \draw (RI)--(SpMid);

        \draw (LI)--(SpBot);
        \draw (LL)--(SpBot);
        \draw (RI)--(SpBot);
        \draw (RL)--(SpBot);

        \def\r{5pt}
        \fill (LO) circle[radius=\r];
        \fill (LU) circle[radius=\r];
        \fill (LI) circle[radius=\r];
        \fill (LL) circle[radius=\r];
        \fill (RU) circle[radius=\r];
        \fill (RI) circle[radius=\r];
        \fill (RL) circle[radius=\r];
        \fill (RO) circle[radius=\r];

        \def\off{3pt} \node[font=\scriptsize,anchor=west, xshift= \off] at (RO) {$a$};
        \node[font=\scriptsize,anchor=north,yshift=-\off] at (RI) {$b$};
        \node[font=\scriptsize,anchor=east, xshift=-\off] at (LO) {$c$};
        \node[font=\scriptsize,anchor=north,yshift=-\off] at (LI) {$d$};
        \node[font=\scriptsize,anchor=south,yshift= \off] at (LU) {$e$};
        \node[font=\scriptsize,anchor=north,yshift=-\off] at (LL) {$f$};
        \node[font=\scriptsize,anchor=south,yshift= \off] at (RU) {$g$};
        \node[font=\scriptsize,anchor=north,yshift=-\off] at (RL) {$h$};
      \end{tikzpicture}
    \end{minipage}
    &
    \begin{minipage}{.45\hsize}
      \begin{equation*}
        \bordermatrix{
          & a & b & c & d & e & f & g & h \cr
          & 2 & 0 & 1 & 0 & 0 & 1 & 0 & 1 \cr
          & 0 & 2 & 0 & 1 & 0 & 1 & 0 & 3 \cr
          & 2 & 2 & 0 & 0 & 1 & 1 & 0 & 2 \cr
          & 2 & 2 & 0 & 0 & 0 & 0 & 1 & 1 
        } \quad \text{over } \mathbb{Z}_{4}
      \end{equation*}
    \end{minipage}
  \end{tabular}
\end{table}

\subsection*{$\AG(3, 2)^{\prime}$}

A smallest non-representable and non-algebraic matroid.
$4$-point planes are the six faces of the cube,
the following six diagonal planes:
\begin{equation*}
  \{ a, b, g, h \}, \; \{ c, d, e, f \}, \; \{ a, d, f, g \}, \;
  \{ b, c, e, h \}, \; \{ a, c, e, g \}, \; \{ b, d, f, h \},
\end{equation*}
and one twisted plane, $\{ b, d, e, g \}$.

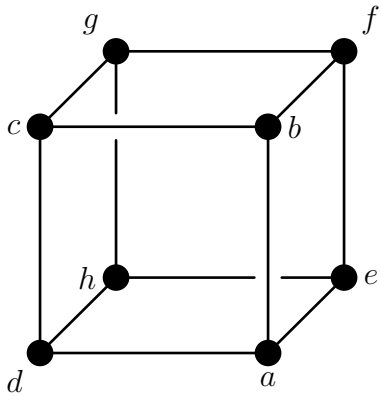
\begin{table}[H]
  \begin{tabular}{cc}
    \begin{minipage}{.45\hsize}
      \centering
      \begin{tikzpicture}[
        x=1cm,y=1cm,
        line cap=round,line join=round,
        line width=1pt
      ]
        \coordinate (v1) at (0,0);
        \coordinate (v2) at (3,0);
        \coordinate (v3) at (4,1);
        \coordinate (v4) at (1,1);
        \coordinate (v5) at (0,3);
        \coordinate (v6) at (3,3);
        \coordinate (v7) at (4,4);
        \coordinate (v8) at (1,4);

        \def\gap{0.18}

        \draw (v1)--(v2);
        \draw (v2)--(v6);
        \draw (v6)--(v5);
        \draw (v5)--(v1);

        \draw (v1)--(v4);
        \draw (v2)--(v3);
        \draw (v5)--(v8);
        \draw (v6)--(v7);

        \draw (v8)--(v7);
        \draw (v7)--(v3);

        \draw (v8)--(1,3+\gap);
        \draw (1,3-\gap)--(v4);

        \draw (v4)--(3-\gap,1);
        \draw (3+\gap,1)--(v3);

        \def\r{5pt}
        \fill (v2) circle[radius=\r] node[below=3pt]       {$a$};
        \fill (v6) circle[radius=\r] node[right=3pt]       {$b$};
        \fill (v5) circle[radius=\r] node[left=3pt]        {$c$};
        \fill (v1) circle[radius=\r] node[below left=2pt]  {$d$};
        \fill (v3) circle[radius=\r] node[right=3pt]       {$e$};
        \fill (v7) circle[radius=\r] node[above right=2pt] {$f$};
        \fill (v8) circle[radius=\r] node[above left=2pt]  {$g$};
        \fill (v4) circle[radius=\r] node[left=3pt]        {$h$};
      \end{tikzpicture}
    \end{minipage}
    &
    \begin{minipage}{.45\hsize}
      \begin{equation*}
        \bordermatrix{
          & a & b & c & d & e & f & g & h \cr
          & 1 & 2 & 0 & 0 & 0 & 0 & 1 & 1 \cr
          & 0 & 2 & 1 & 0 & 0 & 1 & 1 & 0 \cr
          & 0 & 2 & 0 & 1 & 0 & 1 & 0 & 1 \cr
          & 0 & 0 & 0 & 0 & 1 & 1 & 1 & 1 
        } \quad \text{over } \mathbb{Z}_{4}
      \end{equation*}
    \end{minipage}
  \end{tabular}
\end{table} 

\subsection*{Declaration of generative AI and AI-assisted technologies in writing process}

During the preparation of this work, the authors used ChatGPT,
a large language model developed by OpenAI,
in order to improve the grammar, readability, clarity and correctness
of the English in the manuscript and to find and correct typos.
After using this tool, the authors reviewed and edited the content
as needed and take full responsibility for the content of the publication.

\subsection*{Acknowledgements}

This work was supported by the Japan Society for the Promotion of Science
(JSPS) KAKENHI Grants JP25K17298 and JP25K07103.

  \bibliography{reference}
  \bibliographystyle{abbrv}
\end{document}